\input{ifpdf.sty}
\pdffalse

\makeatletter

\newdimen\paperwidth
\newdimen\paperheight

\def\papersize#1#2{\let\p@persize\relax\paperwidth#1\paperheight#2}

\def\Afour{\papersize{210truemm}{297truemm}}

\let\p@persize\Afour

\let\onesidestyle\@twosidefalse
\let\twosidestyle\@twosidetrue

\def\margins{\@ifnextchar[{\@margins}{\@margins[\z@]}}

\def\@margins[#1]#2#3{
  \p@persize\dimen0 #3\dimen0 .5\dimen0\normalsize%
  \oddsidemargin-1truein\advance\oddsidemargin#2%
  \evensidemargin-1truein\advance\evensidemargin#2%
  \topmargin-1truein\advance\topmargin\dimen0\headsep\dimen0\footskip\dimen0%
  \textwidth\paperwidth\advance\textwidth-#2\advance\textwidth-#2%
  \textheight\paperheight\advance\textheight-#3\advance\textheight-#3%
  \headheight\baselineskip\advance\topmargin-.5\baselineskip%
  \advance\headsep-.5\baselineskip%
  \footheight\baselineskip
  \advance\textwidth-#1\advance\oddsidemargin#1
  \if@twoside\def\@themargin%
    {\ifodd\count\z@\oddsidemargin\else\evensidemargin\fi}\fi}

\def\headlinesep#1{\advance\topmargin\headsep\advance\topmargin -#1
  \advance\topmargin.5\baselineskip\headsep #1\advance\headsep-.5\baselineskip}

\def\headline{\if@twoside\let\n@xt\h@dlin@\else\let\n@xt\h@@dlin@\fi\n@xt}
  
\def\h@dlin@#1#2{%
  \def\@oddhead{%
    {{\leftskip\z@\rightskip\z@\noindent\normalsize#1}}}
  \def\@evenhead{%
    {{\leftskip\z@\rightskip\z@\noindent\normalsize#2}}}}

\def\h@@dlin@#1{%
  \def\@oddhead{{{\leftskip\z@\rightskip\z@\noindent\normalsize#1}}}}

\def\footline{\if@twoside\let\n@xt\f@tlin@\else\let\n@xt\f@@tlin@\fi\n@xt}

\def\f@tlin@#1#2{%
  \def\@oddfoot{%
    {{\leftskip\z@\rightskip\z@\noindent\normalsize#1}}}
  \def\@evenfoot{%
    {{\leftskip\z@\rightskip\z@\noindent\normalsize#2}}}}

\def\f@@tlin@#1{%
  \def\@oddfoot{{{\leftskip\z@\rightskip\z@\noindent\normalsize#1}}}}

\def\normalpage{\global\@specialpagefalse}

\makeatother

\makeatletter

\def\ft{\@ifnextchar[{\ft@s}{\ft@}}
\def\ft@{\ft@@@s[\f@size]}
\def\ft@s[{\@ifnextchar{a}{\ft@sz[}{\ft@@s[}}
\def\ft@@s[{\@ifnextchar{s}{\ft@sz[}{\ft@@@s[}}
\def\ft@@@s[#1]{\ft@sz[at #1pt]}
\def\ft@sz[#1]#2{\font\fonttemp=#2 #1\fonttemp\ignorespaces}

\def\pbf{\fontseries\bfdefault\selectfont\boldmath}

\makeatother

\makeatletter

\def\@@bold{bold}

\def\widebar{\ifx\math@version\@@bold
  \let\@widebar\@@@widebar\else\let\@widebar\@@widebar\fi\@widebar}

\def\@@widebar#1{\text{\setbox15\hbox{$#1$}%
  \dimen15 0.45\wd15\advance\dimen15 0.15\ht15%
  \dimen16\ht15\advance\dimen16 0.00em\advance\dimen16 0.3ex%
  \dimen17 0.65\wd15\advance\dimen17 0.05\ht15\advance\dimen17 0.1ex%
  \dimen18 0.035em\advance\dimen18 0.00ex
  \put[\dimen15,\dimen16][c]{\vrule depth 0pt height \dimen18 width \dimen17}}#1}

\def\@@@widebar#1{\text{\setbox15\hbox{$#1$}%
  \dimen15 0.45\wd15\advance\dimen15 0.15\ht15%
  \dimen16\ht15\advance\dimen16 0.00em\advance\dimen16 0.26ex%
  \dimen17 0.65\wd15\advance\dimen17 0.05\ht15\advance\dimen17 0.1ex%
  \dimen18 0.05em\advance\dimen18 0.00ex
  \put[\dimen15,\dimen16][c]{\vrule depth 0pt height \dimen18 width \dimen17}}#1}

\makeatother

\def\smallsquare{\raise-.065em\hbox{$\Box$}}
\def\smallblacksquare{%
  \kern.3ex\vrule depth-.03ex height1.27ex width1.15ex \kern-1.45ex \smallsquare}

\def\smallcircc{\mathop{\mkern3.5mu\text{\raise.58ex\hbox{\ft{lcircle10}a}}}}
\def\varemptyset{{\text{\raise.21ex\hbox{$\not$}}\mkern.15mu\mathrm{O}\mkern.15mu}}

  \let\epsilon\varepsilon
      \let\theta\vartheta
          \let\phi\varphi
   \let\emptyset\varemptyset

\documentstyle[12pt]{article}

\makeatletter

\let\Larg@\Large
\let\hug@\huge

\def\usepackage#1{\input{#1.sty}}

\input{geom.sty}
\input{option_keys.sty}

\let\input\@input

\def\r@adlabel#1#2{\global\@namedef{#1@\the\@key}{#2}}

\let\Large\Larg@
\let\huge\hug@

\def\smallskip{\vskip\smallskipamount}
\def\medskip{\vskip\medskipamount}
\def\bigskip{\vskip\bigskipamount}

\def\mytrivlist{\parsep\parskip\@nmbrlistfalse
  \my@trivlist \labelwidth\z@ \leftmargin\z@
  \itemindent\z@ \def\makelabel##1{##1}}

\def\my@trivlist{\global\@newlisttrue \@outerparskip\parskip}

\def\end#1{\csname end#1\endcsname\@checkend{#1}%
  \expandafter\endgroup\if@endpe\@doendpe\fi
  \if@ignore \global\@ignorefalse \ignorespaces\fi}

\input{epsf.sty}
\usepackage{graphicx}

\def\put{\@ifnextchar[{\@put}{\@@rput[\z@,\z@][r]}}
\def\@put[#1]{\@ifnextchar[{\@@put[#1]}{\@@@@@put[#1]}}
\def\@@put[#1][{\@ifnextchar{l}{\@@lput[#1][}{\@@@put[#1][}}
\def\@@@put[#1][{\@ifnextchar{c}{\@@cput[#1][}{\@@@@put[#1][}}
\def\@@@@put[#1][{\@ifnextchar{r}{\@@rput[#1][}{\relax}}
\def\@@@@@put[{\@ifnextchar{l}{\@@lput[\z@,\z@][}{\@@@@@@put[}}
\def\@@@@@@put[{\@ifnextchar{c}{\@@cput[\z@,\z@][}{\@@@@@@@put[}}
\def\@@@@@@@put[{\@ifnextchar{r}{\@@rput[\z@,\z@][}{\@@@@@@@@put[}}
\def\@@@@@@@@put[#1]{\@@rput[#1][r]}

\let\hm@d@\leavevmode

\long\def\@@lput[#1,#2][l]#3{\setbox0\hbox{#3}\hm@d@\raise#2\hbox to\z@{\dimen0 #1%
  \advance\dimen0-\wd0\kern\dimen0\dp0\z@\ht0\z@\wd0\z@\box0\hss}\ignorespaces}
\long\def\@@cput[#1,#2][c]#3{\setbox0\hbox{#3}\hm@d@\raise#2\hbox to\z@{\dimen0 #1%
  \advance\dimen0-.5\wd0\kern\dimen0\dp0\z@\ht0\z@\wd0\z@\box0\hss}\ignorespaces}
\long\def\@@rput[#1,#2][r]#3{\setbox0\hbox{\kern#1\raise#2\hbox{#3}}%
  \dp0\z@\ht0\z@\wd0\z@\hm@d@\box0\ignorespaces}

\def\flbox{\@ifnextchar[{\@flbox}{\@@rflbox[\z@,\z@][r]}}
\def\@flbox[#1]{\@ifnextchar[{\@@flbox[#1]}{\@@@@@flbox[#1]}}
\def\@@flbox[#1][{\@ifnextchar{l}{\@@lflbox[#1][}{\@@@flbox[#1][}}
\def\@@@flbox[#1][{\@ifnextchar{c}{\@@cflbox[#1][}{\@@@@flbox[#1][}}
\def\@@@@flbox[#1][{\@ifnextchar{r}{\@@rflbox[#1][}{\relax}}
\def\@@@@@flbox[{\@ifnextchar{l}{\@@lflbox[\z@,\z@][}{\@@@@@@flbox[}}
\def\@@@@@@flbox[{\@ifnextchar{c}{\@@cflbox[\z@,\z@][}{\@@@@@@@flbox[}}
\def\@@@@@@@flbox[{\@ifnextchar{r}{\@@rflbox[\z@,\z@][}{\@@@@@@@@flbox[}}
\def\@@@@@@@@flbox[#1]{\@@rflbox[#1][r]}
\long\def\@@lflbox[#1,#2][l]#3{\@@lput[#1,#2][l]{%
  \vtop{\leftskip\z@\parindent\z@\raggedleft\hm@d@#3}}}
\long\def\@@cflbox[#1,#2][c]#3{\@@cput[#1,#2][c]{%
  \vtop{\leftskip\z@\parindent\z@\raggedcenter\hm@d@#3}}}
\long\def\@@rflbox[#1,#2][r]#3{\@@rput[#1,#2][r]{%
  \vtop{\leftskip\z@\parindent\z@\raggedright\hm@d@#3}}}

\def\maketitle{\par
 \begingroup
 \def\thefootnote{\fnsymbol{footnote}}
 \def\@makefnmark{\hbox 
 to 0pt{$^{\@thefnmark}$\hss}} 
 \if@twocolumn 
 \twocolumn[\@maketitle] 
 \else 
 \global\@topnum\z@ \@maketitle \fi\thispagestyle{plain}\@thanks
 \endgroup
 \setcounter{footnote}{0}
 \let\maketitle\relax
 \let\@maketitle\relax
 \gdef\@thanks{}\gdef\@author{}\gdef\@title{}\let\thanks\relax}

\def\@maketitle{ 
 \null
 \vskip 2em \begin{center}
 {\LARGE \@title \par} \vskip 1.5em {\large \lineskip .5em
\begin{tabular}[t]{c}\@author 
 \end{tabular}\par} 
 \vskip 1em {\large \@date} \end{center}
 \par
 \vskip 1.5em}
 
\def\partbeforeskip#1{\def\p@rtbeforeskip{#1}}
\def\partstyle#1{\def\p@rtstyl@{#1}}
\def\partdot#1{\def\partd@t{#1}}
\def\partafterskip#1{\def\p@rtafterskip{#1}}
\def\partintrostyle#1{\def\partintr@styl@{#1}}
\def\partintrodot#1{\def\partintr@dot{#1}}
\long\def\partintrosep#1{\long\def\partintr@sep{#1}}
\def\partnewpagetrue{\def\p@rtnewp@ge{\newpage}}
\def\partnewpagefalse{\long\def\p@rtnewp@ge{\par}}

\partbeforeskip{4ex}
\partstyle{\centering\Large\bf}
\partdot{}
\partafterskip{3ex}
\partintrostyle{\large}
\partintrodot{}
\partintrosep{\par}
\partnewpagefalse

\def\partname{Part}
\def\part{\p@rtnewp@ge\addvspace\p@rtbeforeskip\@afterindentfalse\secdef\@part\@spart}

\def\@part[#1]#2{\ifnum \c@secnumdepth >-1\relax  
        \refstepcounter{part}                     
        \def\@tempa{\addcontentsline{toc}{part}}  %
        \expandafter\@tempa\expandafter{\thepart  
          \hspace{1em}#1}\else                    
        \addcontentsline{toc}{part}{#1}\fi        
   {\p@rtstyl@                       
    \ifnum \c@secnumdepth >-1\relax        
      {\partintr@styl@\partname\ \thepart  
       \partintr@dot}\partintr@sep\nobreak 
    \fi                                    
    #2\partd@t\markboth{}{}\par}
    \nobreak                       
    \vskip\p@rtafterskip           
   \@afterheading                  
    }                              

\def\@spart#1{{\p@rtcentering\p@rtstyl@                      
    #1\partd@t\par}                 
    \nobreak                        
    \vskip\p@rtafterskip            
    \@afterheading                  
  }                                 

\newif\ifsection@ftind
\newif\ifsection@ftpar

\def\sectionbeforeskip#1{\def\s@ctbeforeskip{#1}}
\def\sectionstyle#1{\def\s@ctstyl@{#1}}
\def\sectiondot#1{\def\sectiond@t{#1}}
\def\sectionafterskip#1{\def\s@ctafterskip{#1}}
\def\sectionintrostyle#1{\def\sectionintr@styl@{#1}}
\def\sectionintro#1{\def\sectionintr@{#1}}
\def\sectionintrodot#1{\def\sectionintr@dot{#1}}
\def\sectionintrosep#1{\def\sectionintr@sep{#1}}
\def\sectionindenttrue{\def\s@ctind{\parindent}}
\def\sectionindentfalse{\def\s@ctind{\z@}}
\def\sectionafterindenttrue{\section@ftindtrue}
\def\sectionafterindentfalse{\section@ftindfalse}
\def\sectionafternewlinetrue{\section@ftpartrue}
\def\sectionafternewlinefalse{\section@ftparfalse}

\newif\ifsubsection@ftind
\newif\ifsubsection@ftpar

\def\subsectionbeforeskip#1{\def\ss@ctbeforeskip{#1}}
\def\subsectionstyle#1{\def\ss@ctstyl@{#1}}
\def\subsectiondot#1{\def\subsectiond@t{#1}}
\def\subsectionafterskip#1{\def\ss@ctafterskip{#1}}
\def\subsectionintrostyle#1{\def\subsectionintr@styl@{#1}}
\def\subsectionintro#1{\def\subsectionintr@{#1}}
\def\subsectionintrodot#1{\def\subsectionintr@dot{#1}}
\def\subsectionintrosep#1{\def\subsectionintr@sep{#1}}
\def\subsectionindenttrue{\def\ss@ctind{\parindent}}
\def\subsectionindentfalse{\def\ss@ctind{\z@}}
\def\subsectionafterindenttrue{\subsection@ftindtrue}
\def\subsectionafterindentfalse{\subsection@ftindfalse}
\def\subsectionafternewlinetrue{\subsection@ftpartrue}
\def\subsectionafternewlinefalse{\subsection@ftparfalse}

\newif\ifsubsubsection@ftind
\newif\ifsubsubsection@ftpar

\def\subsubsectionbeforeskip#1{\def\sss@ctbeforeskip{#1}}
\def\subsubsectionstyle#1{\def\sss@ctstyl@{#1}}
\def\subsubsectiondot#1{\def\subsubsectiond@t{#1}}
\def\subsubsectionafterskip#1{\def\sss@ctafterskip{#1}}
\def\subsubsectionintrostyle#1{\def\subsubsectionintr@styl@{#1}}
\def\subsubsectionintro#1{\def\subsubsectionintr@{#1}}
\def\subsubsectionintrodot#1{\def\subsubsectionintr@dot{#1}}
\def\subsubsectionintrosep#1{\def\subsubsectionintr@sep{#1}}
\def\subsubsectionindenttrue{\def\sss@ctind{\parindent}}
\def\subsubsectionindentfalse{\def\sss@ctind{\z@}}
\def\subsubsectionafterindenttrue{\subsubsection@ftindtrue}
\def\subsubsectionafterindentfalse{\subsubsection@ftindfalse}
\def\subsubsectionafternewlinetrue{\subsubsection@ftpartrue}
\def\subsubsectionafternewlinefalse{\subsubsection@ftparfalse}

\newif\ifparagraph@ftind
\newif\ifparagraph@ftpar

\def\paragraphbeforeskip#1{\def\p@rbeforeskip{#1}}
\def\paragraphstyle#1{\def\p@rstyl@{#1}}
\def\paragraphdot#1{\def\paragraphd@t{#1}}
\def\paragraphafterskip#1{\def\p@rafterskip{#1}}
\def\paragraphintrostyle#1{\def\paragraphintr@styl@{#1}}
\def\paragraphintro#1{\def\paragraphintr@{#1}}
\def\paragraphintrodot#1{\def\paragraphintr@dot{#1}}
\def\paragraphintrosep#1{\def\paragraphintr@sep{#1}}
\def\paragraphindenttrue{\def\p@rind{\parindent}}
\def\paragraphindentfalse{\def\p@rind{\z@}}
\def\paragraphafterindenttrue{\paragraph@ftindtrue}
\def\paragraphafterindentfalse{\paragraph@ftindfalse}
\def\paragraphafternewlinetrue{\paragraph@ftpartrue}
\def\paragraphafternewlinefalse{\paragraph@ftparfalse}

\newif\ifsubparagraph@ftind
\newif\ifsubparagraph@ftpar

\def\subparagraphbeforeskip#1{\def\sp@rbeforeskip{#1}}
\def\subparagraphstyle#1{\def\sp@rstyl@{#1}}
\def\subparagraphdot#1{\def\subparagraphd@t{#1}}
\def\subparagraphafterskip#1{\def\sp@rafterskip{#1}}
\def\subparagraphintrostyle#1{\def\subparagraphintr@styl@{#1}}
\def\subparagraphintro#1{\def\subparagraphintr@{#1}}
\def\subparagraphintrodot#1{\def\subparagraphintr@dot{#1}}
\def\subparagraphintrosep#1{\def\subparagraphintr@sep{#1}}
\def\subparagraphindenttrue{\def\sp@rind{\parindent}}
\def\subparagraphindentfalse{\def\sp@rind{\z@}}
\def\subparagraphafterindenttrue{\subparagraph@ftindtrue}
\def\subparagraphafterindentfalse{\subparagraph@ftindfalse}
\def\subparagraphafternewlinetrue{\subparagraph@ftpartrue}
\def\subparagraphafternewlinefalse{\subparagraph@ftparfalse}

\sectionbeforeskip{\bigskipamount}
\sectionstyle{\large\bf}
\sectiondot{}
\sectionafterskip{.5\bigskipamount}
\sectionintrostyle{}
\sectionintro{}
\sectionintrodot{.}
\sectionintrosep{1.25ex}
\sectionindentfalse
\sectionafterindenttrue
\sectionafternewlinetrue

\subsectionbeforeskip{.8\bigskipamount}
\subsectionstyle{\normalsize\bf}
\subsectiondot{}
\subsectionafterskip{.4\bigskipamount}
\subsectionintrostyle{}
\subsectionintro{}
\subsectionintrodot{.}
\subsectionintrosep{1.25ex}
\subsectionindentfalse
\subsectionafterindenttrue
\subsectionafternewlinetrue

\subsubsectionbeforeskip{.6\bigskipamount}
\subsubsectionstyle{\normalsize\bf}
\subsubsectiondot{}
\subsubsectionafterskip{.3\bigskipamount}
\subsubsectionintrostyle{}
\subsubsectionintro{}
\subsubsectionintrodot{.}
\subsubsectionintrosep{1.25ex}
\subsubsectionindentfalse
\subsubsectionafterindenttrue
\subsubsectionafternewlinetrue

\paragraphbeforeskip{.5\bigskipamount}
\paragraphstyle{\normalsize\bf}
\paragraphdot{.}
\paragraphafterskip{1.25ex}
\paragraphintrostyle{}
\paragraphintro{}
\paragraphintrodot{.}
\paragraphintrosep{1.25ex}
\paragraphindentfalse
\paragraphafterindenttrue
\paragraphafternewlinefalse

\subparagraphbeforeskip{.5\bigskipamount}
\subparagraphstyle{\normalsize\bf}
\subparagraphdot{.}
\subparagraphafterskip{1.25ex}
\subparagraphintrostyle{}
\subparagraphintro{}
\subparagraphintrodot{.}
\subparagraphintrosep{1.25ex}
\subparagraphindenttrue
\subparagraphafterindenttrue
\subparagraphafternewlinefalse

\let\@partoken\par
\long\def\@@gobble#1{}
\def\ignorepar{\@ifnextchar\@partoken{\expandafter\ignorepar\@@gobble}{\ignorespaces}}

\def\@startsection#1#2#3#4#5#6{
   \@tempskipa #4\relax
   \csname if#1@ftind\endcsname\@afterindenttrue\else\@afterindentfalse\fi
   \advance\@tempskipa by\presection
   \if@nobreak \everypar{}\else
     \addpenalty{\@secpenalty}\addvspace{\@tempskipa}%
     \allowbreak\vskip -\presection \fi \@ifstar
     {\@ssect{#1}{#2}{#3}{#4}{#5}{#6}}{\@dblarg{\@sect{#1}{#2}{#3}{#4}{#5}{#6}}}}

\def\@sect#1#2#3#4#5#6[#7]#8{\def\object@type{#1}%
   \ifnum #2>\c@secnumdepth\def\@svsec{}\def\@tempb{}%
      \else\refstepcounter{#1}\def\@svsec{{\csname #1intr@styl@\endcsname%
        {\csname #1intr@\endcsname}\csname the#1\endcsname%
        \csname #1intr@dot\endcsname\kern\csname #1intr@sep\endcsname}}%
        \edef\@tempb{\noexpand\numberline{\csname the#1\endcsname}}\fi%
   \def\@tempa{\addcontentsline{toc}{#1}}%
   \csname if#1@ftpar\endcsname%
      \begingroup #6\relax%
        \@hangfrom{\hskip #3\relax\@svsec}{\interlinepenalty \@M{#8}%
        \csname #1d@t\endcsname\par}%
      \endgroup%
      \csname #1mark\endcsname{#7}%
      \expandafter\@tempa\expandafter{\@tempb #7}%
      \ifautolabel\label*{#8}\fi%
   \else%
      \def\@svsechd{#6\hskip #3\relax%
         \@svsec{#8}%
         \csname #1d@t\endcsname%
         \csname #1mark\endcsname{#7}%
         \expandafter\@tempa\expandafter{\@tempb #7}%
         \ifautolabel\label*{#8}\fi}\fi%
   \@xsect{#1}{#5}\ignorepar}

\def\@ssect#1#2#3#4#5#6#7{%
   \ifnum #2>\c@secnumdepth\def\@tempb{}\else \def\@tempb{\numberline{}}\fi%
     \def\@tempa{\addcontentsline{toc}{s#1}}%
     \csname if#1@ftpar\endcsname
        \begingroup #6\relax
           \@hangfrom{\hskip #3}{\interlinepenalty \@M{#7}%
           \csname #1d@t\endcsname\par}%
        \endgroup
        \csname s#1mark\endcsname{#7}%
        \ifstarredcontents\expandafter\@tempa\expandafter{\@tempb #7}\fi%
        \ifautolabel\label*{#7}\fi%
     \else%
        \def\@svsechd{#6\hskip #3\relax{#7}%
        \csname #1d@t\endcsname%
        \csname s#1mark\endcsname{#7}%
        \ifautolabel\label*{#7}\fi}\fi
   \@xsect{#1}{#5}\ignorepar}

\def\@xsect#1#2{
   \csname if#1@ftpar\endcsname 
       \par \nobreak \vskip #2\relax \@afterheading
    \else \global\@nobreakfalse \global\@noskipsectrue
       \everypar{\if@noskipsec \global\@noskipsecfalse
                   \clubpenalty\@M \hskip -\parindent
                   \begingroup \@svsechd \endgroup \unskip
                   \hskip #2\relax  
                  \else \clubpenalty \@clubpenalty
                    \everypar{}\fi}\fi\ignorespaces}

\def\section{\@startsection{section}{1}{\s@ctind}
  {\s@ctbeforeskip}{\s@ctafterskip}{\s@ctstyl@}}
\def\subsection{\@startsection{subsection}{2}{\ss@ctind}
  {\ss@ctbeforeskip}{\ss@ctafterskip}{\ss@ctstyl@}}
\def\subsubsection{\@startsection{subsubsection}{3}{\sss@ctind}
  {\sss@ctbeforeskip}{\sss@ctafterskip}{\sss@ctstyl@}}
\def\paragraph{\@startsection{paragraph}{4}{\p@rind}
  {\p@rbeforeskip}{\p@rafterskip}{\p@rstyl@}}
\def\subparagraph{\@startsection{subparagraph}{4}{\sp@rind}
  {\sp@rbeforeskip}{\sp@rafterskip}{\sp@rstyl@}}

\def\statementabove#1{\def\th@bove{#1}}
\def\statementstyle#1{\def\thstyl@{#1}}
\def\statementbelow#1{\def\thb@low{#1}}
\def\statementindentfalse{\let\thind@nt\relax}
\def\statementindenttrue{\let\thind@nt\indent}

\def\statementintrostyle#1{\def\thintr@style{#1}}
\def\statementintrodot#1{\def\thintr@dot{#1}}
\def\statementintrosep#1{\def\thintr@sep{#1}}
\def\statementintrobrackets#1#2{\def\thintr@left{#1}\def\thintr@right{#2}}

\statementabove{\medskip}
\statementstyle{\sl}
\statementbelow{\medskip}
\statementindenttrue

\statementintrostyle{\normalshape\bf}
\statementintrodot{.}
\statementintrosep{\kern1.25ex}
\statementintrobrackets{(}{)}

\def\@thskip{\dimen100\lastskip\vskip-\dimen100%
  \th@bove\dimen101\lastskip\vskip-\dimen101%
  \ifdim\dimen100>\dimen101\else\dimen100\dimen101\fi\vskip\dimen100\vskip0pt}

\long\def\@@newtheorem#1#2#3{%
  \newenvironment{#3}%
    {\def\object@type{#3}\par\@thskip%
     \@ifnextchar[{\@enva{#3}{\thstyl@#1{#2}}}{\@envb{#3}{\thstyl@#1{#2}}}}%
    {\end{#3@}}%
  \@ifnextchar[{\@nothm{#3}}{\@nnthm{#3}}}

\def\@nothm#1[#2]#3{%
  \@ifundefined{c@#2}{\@latexerr{No theorem environment `#2' defined}\@eha}%
  {\expandafter\@ifdefinable\csname #1@\endcsname
  {\global\@namedef{the#1}{\@nameuse{the#2}}%
   \global\@namedef{c@#1}{\@nameuse{c@#2}}
   \global\@namedef{p@#1}{\@nameuse{p@#2}}
   \global\@namedef{#1@}{\@nnnthm{#2}{#3}}%
   \global\@namedef{end#1@}{\@endtheorem}}}}

\def\@nnnthm#1#2{\refstepcounter
    {#1}\@ifnextchar[{\@ynnnthm{#1}{#2}}{\@xnnnthm{#1}{#2}}}
 
\def\@xnnnthm#1#2{\@begintheorem{#2}{\csname the#1\endcsname}\ignorespaces}
\def\@ynnnthm#1#2[#3]{\@opargbegintheorem{#2}{\csname the#1\endcsname}{#3}\ignorespaces}

\def\renewtheorem{\@ifnextchar[{\@renewtheorem}{\@renewtheorem[{}{}]}}

\long\def\@renewtheorem[#1]{\@@renewtheorem#1}

\long\def\@@renewtheorem#1#2#3{%
  \expandafter\let\csname#3@\endcsname\undefined
  \renewenvironment{#3}%
    {\def\object@type{#3}\par\@thskip%
     \@ifnextchar[{\@enva{#3}{\thstyl@#1{#2}}}{\@envb{#3}{\thstyl@#1{#2}}}}%
    {\end{#3@}}%
  \@ifnextchar[{\@nothm{#3}}{\@nnthm{#3}}}

\def\@begintheorem#1#2{\@opargbegintheorem{#1}{#2}{}}

\def\@opargbegintheorem#1#2#3{%
        \edef\@tempx{#1}%
        \expandafter\let\expandafter\@tempy#2
        \def\@tempz{#3}%
        \mytrivlist\item[\thind@nt\hskip\labelsep%
        {\thintr@style%
          #1\ifx\@tempx\@empty\else\ifx\@tempy\relax\else\kern1ex\fi\fi#2%
          \ifx\@tempz\@empty%
            \ifx\@tempx\@empty\ifx\@tempy\relax%
            \else\thintr@dot\thintr@sep\fi\else\thintr@dot\thintr@sep\fi%
            \else%
            \ifx\@tempx\@empty\ifx\@tempy\relax\else\kern1ex\fi\else\kern1ex\fi%
           \thintr@left{#3}\thintr@right\thintr@dot\thintr@sep\fi}%
            \hskip-\labelsep]%
        \ifautolabel\label*{#3}\fi}

\def\@endtheorem{\endtrivlist\thb@low}

\def\proofname{Proof}

\def\proofabove#1{\def\pf@bove{#1}}
\def\proofstyle#1{\def\pfstyl@{#1}}
\def\proofbelow#1{\def\pfb@low{#1}}
\def\proofindentfalse{\let\pfind@nt\relax}
\def\proofindenttrue{\let\pfind@nt\indent}

\def\proofintrostyle#1{\def\pfintr@style{#1}}
\def\proofintrodot#1{\def\pfintr@dot{#1}}
\def\proofintrosep#1{\def\pfintr@sep{#1}}
\def\proofintrobrackets#1#2{\def\pfintr@left{#1}\def\pfintr@right{#2}}

\proofabove{\medskip}
\proofstyle{}
\proofbelow{\medskip}
\proofindenttrue

\proofintrostyle{\sl}
\proofintrodot{.}
\proofintrosep{\kern1.25ex}
\proofintrobrackets{of\kern1ex}{}

\def\@pfskip{\dimen100\lastskip\vskip-\dimen100%
  \pf@bove\dimen101\lastskip\vskip-\dimen101%
  \ifdim\dimen100>\dimen101\else\dimen100\dimen101\fi\vskip\dimen100\vskip0pt}

\renewenvironment{proof}%
  {\@pfskip\mytrivlist\item[\pfind@nt]\@ifnextchar[{\pro@f}{\pro@f[\prooftag]}}
  {\ifvoid\provedbox\else\hproved\fi\endtrivlist\pfb@low}

\def\pro@f[#1]{\setbox\provedbox\hbox{\provedboxcontents{#1}}\proofintro{#1}}

\def\proofintro#1{\expandafter\def\expandafter\@tempa\expandafter{#1}%
  {\pfintr@style{\proofname\ifx\@tempa\empty\else\kern1ex\pfintr@left{#1}%
  \pfintr@right\fi}\pfintr@dot\pfintr@sep}\pfstyl@\ignorespaces}

\def\provedmark#1{\def\prm@rk{#1}}
\def\provedsep#1{\def\prs@p{#1}}

\provedmark{$\square$}
\provedsep{\kern1.25ex}

\def\provedtexttrue{\def\prb@x##1{\fbox{\small##1}}}
\def\provedtextfalse{\def\prb@x##1{\prm@rk}}
\def\provedmarkrighttrue{\let\prhf@l\hfill}
\def\provedmarkrightfalse{\let\prhf@l\relax}

\provedtextfalse
\provedmarkrightfalse

\def\provedboxcontents#1{\expandafter\def\expandafter\@tempa\expandafter{#1}%
  \ifx\@tempa\empty\prm@rk\else\prb@x{#1}\fi}

\def\proved{\ifmmode\eqno{\box\provedbox}\else\hproved\fi}

\def\hproved{\unskip\nobreak\prhf@l\penalty50\prs@p\hbox{}\nobreak\prhf@l
  \box\provedbox{\finalhyphendemerits=0\par}}

\def\captionstyle#1{\def\c@ptstyl@{#1}}
\def\captionintrostyle#1{\def\c@pintr@style{#1}}
\def\captionintrodot#1{\def\c@pintr@dot{#1}}
\def\captionintrosep#1{\def\c@pintr@sep{#1}}

\captionstyle{\small\sf}
\captionintrostyle{\bf}
\captionintrodot{.}
\captionintrosep{\hskip1.25ex}

\long\def\@makecaption#1#2{%
    \vskip\captionskip
    \setbox\@tempboxa\hbox{%
      \ifproofing\@ifundefined{the@label}{}
        {\hbox to 0pt{\vbox to 0pt{\vss\hbox{\tiny\the@label}\bigskip}\hss}}\fi
      \c@ptstyl@{\c@pintr@style #1\c@pintr@dot}\ignorespaces #2}%
    \@captionwidth=\hsize \advance\@captionwidth-2\@captionmargin
    \ifdim \wd\@tempboxa >\@captionwidth {%
        \rightskip=\@captionmargin\leftskip=\@captionmargin
        \unhbox\@tempboxa\par}%
      \else
        \hbox to\hsize{\hfil\box\@tempboxa\hfil}%
    \fi}

\def\end@Float#1{%
  \expandafter\caption\expandafter[\the@title]{%
   {\c@pintr@style%
   \ifx\the@caption\empty\ifx\the@title\empty
   \else\c@pintr@sep\fi\else\c@pintr@sep\fi
    \the@title\ifx\the@caption\empty%
     \expandafter\label\expandafter*\expandafter{\the@label}%
    \else\ifx\the@title\empty%
     \expandafter\label\expandafter*\expandafter{\the@label}%
    \else\c@pintr@dot\c@pintr@sep%
     \expandafter\label\expandafter*\expandafter{\the@label}\fi\fi}%
   \ignorespaces\the@caption}%
  \end{#1}}

\def\@myFloat#1[#2]#3{%
  \def\color@hbox{}\def\color@vbox{}\def\color@endbox{}%
  \begin{#1}[#2]\def\the@label{#3}}

\def\fig#1{\@ifnextchar[{\@fig{#1}}{\@fig{#1}[0pt]}}

\def\@fig#1[#2]#3{\@ifnextchar[{\@@fig{#1}[#2]{#3}}{\@@fig{#1}[#2]{#3}[0pt]}}
\def\@@fig#1[#2]#3[#4]#5#6{%
  \def\the@title{#5}\def\the@caption{#6}\centerline{\fig@{#1}{#2}{#3}}\vskip#4}

\def\fig@@#1#2#3{\leavevmode{\figstyl@\vrule width 0pt height 1.8ex%
 \smash{\framebox{\strut\def\@temp{#1}\ifx\@temp\@empty{ #3 }\else{ #1 }\fi}}}}
\def\fig@@@#1#2#3{\leavevmode\kern#2\epsfbox{#3}}

\def\figstyle#1{\def\figstyl@{#1}}

\figstyle{\fontfamily{cmr}\fontshape{n}\fontseries{m}\selectfont\normalsize}

\newcounter{diagram}

\let\thediagram\theequation
\def\ftype@diagram{2}
\def\ext@diagram{lod}
\def\diagram{\@float{diagram}}
\let\enddiagram\end@float
\newif\if@diagnum

\def\diag#1{\@ifnextchar[{\@diag{#1}}{\@diag{#1}[0pt]}}

\def\@diag#1[#2]#3{\@ifnextchar[{\@@diag{#1}[#2]{#3}}{\@@diag{#1}[#2]{#3}[0pt]}}

\def\@@diag#1[#2]#3[#4]#5{
  \def\the@tag{#5}\@eqnswtrue%
  \centerline{\setbox0\hbox{\diag@{#1}{#2}{#3}}
  \dimen0 -0.5\wd0\dimen1 0.5\ht0\box0%
  \advance\dimen0 0.5\hsize\advance\dimen0 -\rightskip\advance\dimen1 #4%
  \let\@currentlabel\the@tag%
  \setbox0\hbox to 0pt{\hss%
    \fontfamily{cmr}\fontshape{n}\fontseries{m}\selectfont(\the@tag)}%
  \ifx\the@tag\@empty\refstepcounter{equation}\let\@currentlabel\theequation%
    \setbox0\hbox to 0pt{\hss%
      \fontfamily{cmr}\fontshape{n}\fontseries{m}\selectfont(\thediagram)}\fi%
  \if@eqnsw\else\let\@currentlabel\relax\setbox0\hbox to 0pt{}\fi%
  \advance\dimen1 -0.5\ht0%
  \put[\dimen0,\dimen1][l]{%
    \box0\expandafter\label\expandafter*\expandafter{\the@label}\kern0.15em}}}

\def\diag@@#1#2#3{\leavevmode{\diagstyl@\vrule width 0pt height 1.8ex%
 \smash{\framebox{\strut\def\@temp{#1}\ifx\@temp\@empty{ #3 }\else{ #1 }\fi}}}}
\def\diag@@@#1#2#3{\leavevmode\kern#2\epsfbox{#3}}

\def\diagstyle#1{\def\diagstyl@{#1}}

\diagstyle{\fontfamily{cmr}\fontshape{n}\fontseries{m}\selectfont\normalsize}

\def\showfiguresfalse{\let\fig@\fig@@}
\def\showfigurestrue{\let\fig@\fig@@@}

\def\showdiagramsfalse{\let\diag@\diag@@}
\def\showdiagramstrue{\let\diag@\diag@@@}

\showfigurestrue
\showdiagramstrue

\def\n@number{\@eqnswfalse\let\@currentlabel\relax\let\the@tag\relax}

\def\equation{$$
  \@eqnswtrue\def\object@type{equation}\let\nonumber\n@number%
  \advance\c@equation1\edef\@currentlabel{\theequation}\advance\c@equation-1%
  \def\the@tag{\refstepcounter{equation}\eqno\hbox{\@eqnnum}}}

\def\tag#1{\edef\@currentlabel{#1}\def\the@tag{\eqno\hbox{\reset@font\rm(#1)}}}

\def\endequation{\the@tag$$
  \global\@ignoretrue}

\let\it@m\item
\def\item{\@ifnextchar[{\item@}{\item@@}}
\def\item@[#1]{\it@m[#1]\vskip-\lastskip\vskip\itemsep}
\def\item@@{\it@m\vskip-\lastskip\vskip\itemsep}
\def\s@titemsep{\@ifnextchar[{\s@@titemsep}{\relax}}
\def\s@@titemsep[#1]{\itemsep#1}

\let\@itemize\itemize
\let\@enditemize\enditemize
\renewenvironment{itemize}
{\@itemize\itemsep3pt\parsep0pt\topsep0pt\partopsep0pt\s@titemsep}
{\@enditemize\vskip-\lastskip\vskip\itemsep}

\let\@enumerate\enumerate
\let\@endenumerate\endenumerate
\renewenvironment{enumerate}
{\@enumerate\itemsep3pt\parsep0pt\topsep0pt\partopsep0pt\s@titemsep}
{\@endenumerate\vskip-\lastskip\vskip\itemsep}

\let\@description\description
\let\@enddescription\enddescription
\renewenvironment{description}
{\@description\itemsep3pt\parsep0pt\topsep0pt\partopsep0pt\s@titemsep}
{\@enddescription\vskip-\lastskip\vskip\itemsep}

\topsep\dimen200
\partopsep\dimen201

\@definecounter{bibenumi}

\def\thebibliography#1{%
 \section*{\refname}\vskip-\lastskip%
 \list{[\arabic{bibenumi}]}{\topsep0pt\settowidth\labelwidth{[#1]}%
 \leftmargin\labelwidth\advance\leftmargin\labelsep\usecounter{bibenumi}}%
 \def\newblock{\hskip .11em plus .33em minus .07em}%
 \sloppy\clubpenalty4000\widowpenalty4000\sfcode`\.=1000\relax}

\let\@ref@\ref
\let\@pageref@\pageref
\let\@fullref@\fullref
\let\@Fullref@\Fullref
\let\@reftype@\reftype
\let\@Reftype@\Reftype
\let\@label@\label
\let\@cite@\cite
\let\@bibitem@\bibitem

\def\label{\@ifnextchar*{\label@}{\label@{}}}
\def\label@#1#2{\@label@#1{#2}\putl@bel{#2}\ignorespaces}
\def\putl@b@l#1{\put[0pt,.25\baselineskip]{%
  \hbox{\labc@lor{\fontfamily{cmr}\fontshape{n}\fontseries{m}\selectfont%
  \tiny\setbox5\hbox{\vphantom{X}\smash{\ns#1}}%
  \hbox to 0pt{\hss\tiny$\blacktriangledown$\kern-.085em}%
  \raise2.25ex\hbox to 0pt{\hss\framebox{\box5}}}}}}

\def\putr@fl@bel#1{{\let\labc@lor\refc@lor\putl@bel{#1}}}
\def\ref@#1{\@ref@{#1}\putr@fl@bel{#1}}
\def\pageref@#1{\@pageref@{#1}\putr@fl@bel{#1}}
\def\fullref@#1{\@fullref@{#1}\putr@fl@bel{#1}}
\def\Fullref@#1{\@Fullref@{#1}\putr@fl@bel{#1}}
\def\reftype@#1{\@reftype@{#1}\putr@fl@bel{#1}}
\def\Reftype@#1{\@Reftype@{#1}\putr@fl@bel{#1}}

\def\ref@@#1{\leavevmode\refc@lor{\rm$\langle$#1$\rangle$}}
\let\pageref@@\ref@@
\let\Fullref@@\ref@@
\let\fullref@@\ref@@
\let\reftype@@\ref@@
\let\Reftype@@\ref@@

\def\bibitem{\@ifnextchar[{\bibitem@@}{\bibitem@@@}}
\def\bibitem@@[#1]#2{\@bibitem@[#1]{#2}\putl@bel{#2}}
\def\bibitem@@@#1{\@bibitem@{#1}\putl@bel{#1}\ignorespaces}

\def\cit@{\@ifnextchar[{\@cit@@@}{\@cit@@}}
\def\@cit@@#1{\@cite@{#1}{\let\labc@lor\citc@lor\putl@bel{#1}}}
\def\@cit@@@[#1]#2{\@cite@[#1]{#2}{\let\labc@lor\citc@lor\putl@bel{#2}}}
\def\cit@@{\@ifnextchar[{\cit@@@@}{\cit@@@}}
\def\cit@@@#1{\leavevmode{\citc@lor\rm[#1]}}
\def\cit@@@@[#1]#2{\leavevmode{\citc@lor\rm[#2, #1]}}

\def\showcitationstrue{\let\cite\cit@}
\def\showcitationsfalse{\let\cite\cit@@}

\def\showreferencestrue{%
  \let\ref\ref@\let\pageref\pageref@%
  \let\fullref\fullref@\let\Fullref\Fullref@%
  \let\reftype\reftype@\let\Reftype\Reftype@}
\def\showreferencesfalse{%
  \let\ref\ref@@\let\pageref\pageref@@%
  \let\fullref\fullref@@\let\Fullref\Fullref@@%
  \let\reftype\reftype@@\let\Reftype\Reftype@@}

\def\showlabelstrue{\let\putl@bel\putl@b@l}
\def\showlabelsfalse{\let\putl@bel\hid@@}

\def\postit@{\@ifnextchar[{\postit@@}{\p@tp@stit}}
\def\postit@@[#1]{\postit@@@#1,@}
\def\postit@@@#1,{\@ifnextchar{@}{\p@@tp@stit{#1}}{\postit@@@@#1,}}
\def\postit@@@@#1,#2,@{\p@@@tp@stit{#1}{#2}}

\long\def\p@tp@stit#1{\put[0pt,.25\baselineskip]{%
  \hbox{\postitc@lor{\fontfamily{cmr}\fontshape{n}\fontseries{m}\selectfont%
  \tiny\setbox5\hbox{\vphantom{X}\smash{\ns#1}}%
  \hbox to 0pt{\hss\tiny$\blacktriangledown$\kern-.085em}%
  \raise2.25ex\hbox to 0pt{\hss\framebox{\box5}}}}}}

\long\def\p@@tp@stit#1@#2{\put[0pt,.25\baselineskip]{%
  \hbox{\postitc@lor{\fontfamily{cmr}\fontshape{n}\fontseries{m}\selectfont%
  \tiny\setbox5\hbox{\vbox{\hsize#1\leftskip\z@\raggedright
  \parindent\z@{\ns#2\par}\vss}}%
  \hbox to 0pt{\hss\tiny$\blacktriangledown$\kern-.085em}%
  \raise2.25ex\hbox to 0pt{\hss\framebox{\box5}}}}}}

\long\def\p@@@tp@stit#1#2#3{\put[0pt,.25\baselineskip]{%
  \hbox{\postitc@lor{\fontfamily{cmr}\fontshape{n}\fontseries{m}\selectfont%
  \tiny\setbox5\hbox{\vbox to #2{\hsize#1\leftskip\z@\raggedright
  \parindent\z@{\ns#3\par}\vss}}%
  \hbox to 0pt{\hss\tiny$\blacktriangledown$\kern-.085em}%
  \raise2.25ex\hbox to 0pt{\hss\framebox{\box5}}}}}}

\def\postitc@lor{\color{postitcolor}}

\def\showpostittrue{\let\postit\postit@}
\def\showpostitfalse{\let\postit\hid@@@}

\long\def\hid@@#1{\ignorespaces}
\def\hid@@@{\@ifnextchar[{\hid@@@@}{\hid@@}}
\long\def\hid@@@@[#1]{\hid@@}

\makeatother

\papersize{216truemm}{279truemm}
\margins{25.5truemm}{14.5truemm}
\advance\voffset-3truemm
\advance\hoffset3truemm
\advance\footskip2truemm

\headline{\hfill}
\footline{\small\hfill--\kern1ex\thepage\kern1ex--\hfill}

\sectionbeforeskip{1.5\bigskipamount}
\sectionstyle{\centering\normalsize\sc}
\sectionafterskip{\bigskipamount}
\sectionafterindenttrue

\subsectionbeforeskip{\bigskipamount}
\subsectionstyle{\normalsize\bf}
\subsectionafterskip{.5\bigskipamount}
\subsectionafterindenttrue

\paragraphbeforeskip{\medskipamount}
\paragraphstyle{\ft{cmbxsl10}\boldmath}
\paragraphafterskip{1.25ex}
\paragraphindenttrue
\paragraphdot{.}

\abovedisplayskip\smallskipamount
\belowdisplayskip\smallskipamount
\abovedisplayshortskip\smallskipamount
\belowdisplayshortskip\smallskipamount


\textfloatsep\floatsep

\proofingfalse
\autolabelfalse

\showlabelsfalse
\showcitationstrue
\showreferencestrue

\showfigurestrue
\showdiagramstrue

\newtheorem{stat}{\statname}  \unnumbered{stat}
\newenvironment{statement}[1]{\def\statname{#1}\begin{stat}}{\end{stat}}

\newtheorem{nstat}{\nstatname}[section]

\newtheorem[{\ns}{}]{definition}[nstat]{Definition}
\newtheorem{lemma}[nstat]{Lemma}
\newtheorem{proposition}[nstat]{Proposition}
\newtheorem{theorem}[nstat]{Theorem}
\newtheorem{corollary}[nstat]{Corollary}

\newtheorem{question}[nstat]{Question}
\newtheorem[{\ns}{}]{exercise}[nstat]{Exercise}
\newtheorem[{\ns}{}]{example}[nstat]{Example}
\newtheorem[{\ns}{}]{remark}[nstat]{Remark}

\let\ns\normalshape

\captionstyle{}
\captionintrostyle{}

\makeatletter

\def\@@bold{bold}

\def\widebar{\ifx\math@version\@@bold
  \let\@widebar\@@@widebar\else\let\@widebar\@@widebar\fi\@widebar}

\def\@@widebar#1{\text{\setbox15\hbox{$#1$}%
  \dimen15 0.45\wd15\advance\dimen15 0.15\ht15%
  \dimen16\ht15\advance\dimen16 0.00em\advance\dimen16 0.3ex%
  \dimen17 0.65\wd15\advance\dimen17 0.05\ht15\advance\dimen17 0.1ex%
  \dimen18 0.035em\advance\dimen18 0.00ex
  \put[\dimen15,\dimen16][c]{\vrule depth 0pt height \dimen18 width \dimen17}}#1}

\def\@@@widebar#1{\text{\setbox15\hbox{$#1$}%
  \dimen15 0.45\wd15\advance\dimen15 0.15\ht15%
  \dimen16\ht15\advance\dimen16 0.00em\advance\dimen16 0.26ex%
  \dimen17 0.65\wd15\advance\dimen17 0.05\ht15\advance\dimen17 0.1ex%
  \dimen18 0.05em\advance\dimen18 0.00ex
  \put[\dimen15,\dimen16][c]{\vrule depth 0pt height \dimen18 width \dimen17}}#1}

\makeatother

\showfigurestrue

\sectionbeforeskip{1.5\bigskipamount}
\sectionstyle{\centering\normalsize\bf}
\sectionafterskip{\bigskipamount}
\sectionafterindenttrue

\subsectionbeforeskip{\bigskipamount}
\subsectionstyle{\normalsize\bf}
\subsectionafterskip{.5\bigskipamount}
\subsectionafterindenttrue

\paragraphbeforeskip{\bigskipamount}
\paragraphstyle{\centering\normalsize\sl}
\paragraphafterskip{.5\bigskipamount}
\paragraphafternewlinetrue
\paragraphafterindenttrue
\paragraphdot{}

\statementindenttrue
\statementintrostyle{\sc}
\renewtheorem[{\sl}{}]{theorem}{Theorem}[section]
\renewtheorem[{\sl}{}]{corollary}[theorem]{Corollary}
\renewtheorem[{\sl}{}]{question}[theorem]{Question}
\renewtheorem[{\sl}{}]{problem}[theorem]{Problem}
\renewtheorem[{\sl}{}]{lemma}[theorem]{Lemma}
\renewtheorem[{\ns}{}]{definition}[theorem]{Definition}
\renewtheorem[{\ns}{}]{remark}[theorem]{Remark}

\captionstyle{\small}
\captionintrostyle{\sc}

\newcommand{\id}{\mathop{\mathrm{id}}\nolimits}
\newcommand{\Cl}{\mathop{\mathrm{Cl}}\nolimits} 
\newcommand{\Int}{\mathop{\mathrm{Int}}\nolimits} 
\newcommand{\Bd}{\partial} 

\newcommand{\cs}{\mathop{\#}}

\renewcommand{\:}{\,{:}\;}
\newcommand{\simtimes}{\mathbin{\widetilde{\smash{\times}}}}
\newcommand{\CP}{{C\mkern-1.5muP}}
\newcommand{\CPbar}{{\vphantom{CP}\smash{\widebar{C\mkern-1.5muP}}}}
\newcommand{\Bbar}{{\vphantom{B}\smash{\widebar{B}}}}
\newcommand{\Tor}{\mathop{\mathrm{Tor}}\nolimits}
\newcommand{\F}{{\mathbb F}}
\newcommand{\Z}{{\mathbb Z}}

\newcommand{\R}{{\mathbb R}}

\def\(#1\){$(${\sl #1}\/$)$}

\def\varemptyset{{\text{\raise.21ex\hbox{$\not$}}\mkern.15mu\mathrm{O}\mkern.15mu}}

\let\emptyset\varemptyset

\def\emph#1{{\sl #1}\/}

\begin{document}

\title{\large\bf BRANCHED COVERINGS OF $\CP^2$\\%
                 AND OTHER BASIC 4-MANIFOLDS}
\author{
\normalsize \sc Riccardo Piergallini\\
\normalsize \sl Scuola di Scienze e Tecnologie\\[-3pt]
\normalsize \sl Universit\`a di Camerino -- Italy\\
\small \tt riccardo.piergallini@unicam.it
\and
\normalsize \sc Daniele Zuddas\\
\normalsize \sl Dipartimento di Matematica e Geoscienze\\[-3pt]
\normalsize \sl Universit\`a di Trieste -- Italy\\
\small \tt dzuddas@units.it
}
\date{}

\vglue-0pt
\maketitle
\vskip-5pt\vskip0pt

\begin{abstract}\noindent
We give necessary and sufficient conditions for a 4-manifold to be a branched covering of
$\CP^2$, $S^2\times S^2$, $S^2 \simtimes S^2$ or $S^3 \times S^1$, which are expressed in
terms of the Betti numbers and the signature of the 4-manifold. Moreover, we extend these
results to include branched coverings of connected sums of the above manifolds. This leads
to some new examples of closed simply connected quasiregularly elliptic 4-manifolds.
\smallskip

\medskip\smallskip\noindent
{\sl Keywords}\/: branched covering, 4-manifold, 2-knot, surface knot.

\medskip\noindent
{\sl AMS Classification}\/: 57M12 (primary), 57K40, 57K45 (secondary).
\end{abstract}

\section{Introduction}

In \cite{Pi95} the first author proved that any closed orientable PL 4-manifold $M$ is a
simple\break 4-fold covering of $S^4$ branched over a closed locally flat PL surface
self-transversally\break immersed in $S^4$. Subsequently, in \cite{IP02} the self-intersections
of the branch surface were shown to be removable once the covering has been stabilized to
degree five, obtaining $M$ as a 5-fold covering of $S^4$ branched over a closed locally
flat PL surface embedded in $S^4$.

On the other hand, it is a classical fact of algebraic geometry that any smooth
irreducible projective surface $S \subset \CP^n$ is a holomorphic branched covering of
$\CP^2$ obtained by taking a general projection, where the branch set is an irreducible
nodal cuspidal algebraic curve in $\CP^2$. Even though this result is folklore, a proof
appeared surprisingly only in a 2011 paper by Ciliberto and Flamini \cite{CF11}.

Furthermore, Auroux in \cite{Au00} extended this result to all closed integral symplectic\break
4-manifolds $M$, proving that, roughly, they are realizable as ``symplectic'' coverings of
$\CP^2$ branched over a symplectic nodal cuspidal surface in $\CP^2$. In fact, every
closed integral symplectic 4-manifold $(M,\omega)$ admits a branched covering $M \to
\CP^2$, such that the pullback of the Fubini-Study form $\omega_{\text{FS}}$ can be
suitably perturbed to a symplectic form which is ambient isotopic to $k \omega$ for some
sufficiently large integer $k$. Moreover, any symplectic form on $M$ is homotopic, through
symplectic forms, to an integral one realizable as above. It is worth noting that there is
a subtle difference between holomorphic and symplectic singular surfaces in $\CP^2$:
holomorphic nodal singularities are always positive, while the symplectic ones may also be
negative.

Hence, it is interesting to study the topology of branched coverings of $\CP^2$, and a
natural question is the following: \emph{which closed oriented 4-manifolds are realizable
as branched coverings of $\CP^2$?}

In this paper we give a complete answer to this question, by proving that a closed
connected orientable PL $4$-manifold $M$ is a simple branched covering of $\CP^2$
(branched over an embedded locally flat surface) if and only if the second Betti number
$b_2(M)$ is positive. In addition, we also characterize the 4-manifolds $M$ that are
branched coverings of $S^2 \times S^2$, $S^2 \simtimes S^2$ and $S^3 \times S^1$. Finally,
we generalize these results to branched coverings of $\cs_m \CP^2 \cs_n\CPbar^2$,
$\cs_n(S^2\times S^2)$ and $\cs_n(S^3 \times S^1)$.

The proofs of all these results follow the same idea: we split $M$ into two pieces, based
on certain submanifolds $N \subset M$, and represent them as branched coverings of
standard bounded 4-manifolds by using \cite{PZ16}, then we glue such branched coverings
together. As a consequence of this argument, we also obtain a representation of the
submanifolds $N \subset M$ as branched coverings of suitable standard submanifolds of the
base spaces considered above, see Section \ref{BC-sub/sec}.

For the sake of convenience, we work in the PL category. Nevertheless, our results can be
easily translated into the smooth category as well, being $\text{PL} = \text{Diff}$ in
dimension four.

The following notations is used throughout the paper: $\CP^2$ and $\CPbar^2$ for the
complex pro\-jec\-tive space with the standard and the opposite orientation, respectively;
$S^2 \simtimes S^2 \cong \CP^2 \mathbin{\#} \CPbar^2$ for the twisted $S^2$-bundle over
$S^2$; $b_i(M)$ for the $i$-th Betti number of $M$; 
$$\beta_M \: H_2(M)/\Tor H_2(M) \times H_2(M)/\Tor H_2(M) \to \Z$$ 
for the intersection form of $M$; $b_2^+(M)$ (resp. $b_2^-(M)$) for the maximal dimension
of a vector subspace of $H_2(M; \R)$ where $\beta_M$ is positive (resp. negative)
definite; and finally $\sigma(M)$ for the signature of $M$ (see \cite{GS99}, \cite{Ki89}
or \cite{MH73}). Now we state our main theorems.

\begin{theorem}\label{main/thm}
Let $M$ be a closed connected oriented PL 4-manifold. Then, there exists a branched 
covering $p\: M \to N$ with:
\begin{itemize}
\item[\(a\)] $N = \CP^2$ $\Leftrightarrow$ $b_2^+(M) \geq 1;$
\item[\(b\)] $N = \CPbar^2$ $\Leftrightarrow$ $b_2^-(M) \geq 1;$
\item[\(c\)] $N = S^2 \simtimes S^2$ $\Leftrightarrow$ $b_2^+(M) \geq 1$ and $b_2^-(M)
\geq 1;$
\item[\(d\)] $N = S^2 \times S^2$ $\Leftrightarrow$ $b_2^+(M) \geq 1$ and $b_2^-(M) 
\geq 1;$
\item[\(e\)] $N = S^3 \times S^1$ $\Leftrightarrow$ $b_1(M) \geq 1.$
\end{itemize}
In all cases, we can assume that $p$ is a simple branched covering of degree $d\leq4$, 
whose branch set $B_p$ is a closed locally flat PL surface self-transversally immersed in 
$N$.
Moreover, $B_p$ can be desingularized to become embedded in $N$, with the following
estimates for the degree $d$: $d \leq 5$ in cases \(a\) and \(b\) for $b_2(M) \geq 2$ and
$\beta_M$ odd, case \(c\) for $\beta_M$ odd, case \(d\) for $\beta_M$ even, and case
\(e\); $d \leq 6$ in cases \(a\) and \(b\) for $b_2(M) \geq 2$ and $\beta_M$ even, case
\(c\) for $\beta_M$ even, and case \(d\) for $\beta_M$ odd; $d \leq 9$ in cases \(a\) and
\(b\) for $b_2(M) = 1$.
\end{theorem}

\begin{remark}
If $\beta_M$ is indefinite, then $M$ is a simple branched covering of all of $\CP^2$,
$\CPbar^2$, $S^2 \simtimes S^2$ and $S^2 \times S^2$. On the other hand, if $\beta_M$ is
positive (resp. negative) definite, then among these manifolds $\CP^2$ (resp. $\CPbar^2$)
is the only one of which $M$ is a branched covering.
\end{remark}

For the sake of completeness, we also state the following generalization of Theorem 
\ref{main/thm}. The proof is based on the same methods of that of Theorem \ref{main/thm}, 
and we will only sketch it.

\begin{theorem}\label{main-gen/thm}
Let $M$ be a closed connected oriented PL 4-manifold and let $m$ and $n$ be non-negative 
integers. Then, there exists a branched covering $p\: M \to N$ with:
\begin{itemize}
\item[\(a\)] $N = \cs_m\CP^2 \cs_n \CPbar^2$ $\Leftrightarrow$ $b_2^+(M) \geq m$ and 
$b_2^-(M) \geq n$;
\item[\(b\)] $N = \cs_n(S^2 \times S^2)$ $\Leftrightarrow$ $b_2^+(M)\geq n$ and $b_2^-(M) 
\geq n$;
\item[\(c\)] $N = \cs_n(S^3 \times S^1)$ $\Leftrightarrow$ $\pi_1(M)$ admits a free group 
of rank $n$ as a quotient.
\end{itemize}
In all cases, we can assume that $p$ is a simple branched covering of degree $d\leq4$, 
whose branch set $B_p$ is a closed locally flat PL surface self-transversally immersed in 
$N$.
Moreover, $B_p$ can be desingularized to become embedded in $N$, with the following
estimates for the degree $d$: $d \leq 5$ in case \(a\) for $b_2(M) \geq 2(m+n)$ and
$\beta_M$ odd, case \(b\) for $\beta_M$ even, and case \(c\); $d \leq 6$ in case \(a\) for
$b_2(M) \geq 2(m+n)$ and $\beta_M$ even, and case \(b\) for $\beta_M$ odd; $d \leq 9$ in
case \(a\) for $b_2(M) < 2(m+n)$.
\end{theorem}

We observe that Theorem \ref{main-gen/thm} \(a\) includes Theorem \ref{main/thm} \(a\),
\(b\) and \(c\), being $S^2 \simtimes S^2 \cong \CP^2 \cs \CPbar^2$. Similarly, it
includes the case of $N = \cs_m(S^2 \times S^2) \cs_n(S^2 \simtimes S^2)$ with $n \geq 1$,
being $(S^2 \times S^2) \cs \CP^2 \cong (S^2 \simtimes S^2) \cs \CP^2$.

\begin{remark}
As a consequence of Theorem \ref{main-gen/thm} \(a\) and \(b\), we obtain some simply
connected 4-manifolds $N$ admitting a simple branched covering $p \: T^4 \to N$. Namely,
they are $\cs_m\CP^2 \cs_n \CPbar^2$ and $\cs_n(S^2 \times S^2)$ for any $m \leq 3$ and $n
\leq 3$. This extends the previous result by Rickman \cite{Ri06} concerning the case when
$N$ is $\cs_2(S^2 \times S^2)$. All such manifolds $N$ are \emph{quasiregularly elliptic}
(see Bonk and Heinonen \cite{BH01} for the definition), since the composition of the
universal covering of $T^4$ with $p$ is a \emph{quasiregular map} $\R^4 \to N$. The
question of which closed simply connected manifolds are quasiregularly elliptic was posed
by Gromov in \cite{Gr81,Gr07}. According to Prywes \cite{Pr19}, $b_2(M) \leq 6$ for any
closed connected orientable quasiregularly elliptic 4-manifold $M$, in particular
$\cs_n(S^2 \times S^2)$ is not quasiregularly elliptic for $n \geq4$. Hence our result
implies a sharp answer to the Gromov question for such connected sums, while the cases of
$\cs_m\CP^2 \cs_n \CPbar^2$ with $m + n \leq 6$ and $\max(m,n) \geq 4$, as well as the
exotic counterparts of all the above manifolds, remain still open.
\end{remark}

It is known that there are smooth 4-manifolds $X_{m,n}$ homeomorphic but not diffeomorphic
to $\cs_m\CP^2 \cs_n \CPbar^2$ for certain $m, n \geq 1$, see for example Donaldson
\cite{Do87-1}, Akhmedov and Park \cite{AP08,AP10} and Park, Stipsicz and Szab\'o
\cite{PSS05}. As an immediate consequence of Theorem \ref{main-gen/thm}, we get the
following corollary.

\begin{corollary}
For every smooth 4-manifold $X_{m,n}$ homeomorphic to $\cs_m\CP^2 \cs_n \CPbar^2$, there
exists a smooth simple covering $p \: X_{m,n} \to \cs_m\CP^2 \cs_n \CPbar^2$ of degree
$\leq 4$ $($resp. $\leq 9)$ branched over a smooth self-transversally immersed $($resp.
embedded\/$)$ surface.
\end{corollary}

\section{Preliminaries}

We briefly recall the notion of branched covering, in order to introduce some terminology
(see \cite{BP05} or \cite{GS99} for more details).

A map $p \: M \to N$ between compact \emph{oriented} PL manifolds having the same
dimension $n$ is called a \emph{branched covering} if it is a non-degenerate
\emph{orientation preserving} PL map with the following properties: 1)~there is an
$(n-2)$-di\-mensional polyhedral subspace $B_p \subset N$, the \emph{branch set} of $p$,
such that the restriction $p_{|} \: M - p^{-1}(B_p) \to N - B_p$ is an ordinary covering
of finite \emph{degree} $d(p)$ (we assume $B_p$ to be minimal with respect to this
property); 2)~in the bounded case, $p^{-1}(\Bd N) = \Bd M$ and $p$ preserves the product
structure of a collar of the boundaries (which implies that the restriction to the
boundary $p_{|}\: \Bd M \to \Bd N$ is a branched covering of the same degree of $p$).

Moreover, $p$ is called \emph{simple} if the monodromy of the above mentioned ordinary
covering sends every meridian around $B_p$ to a transposition. In this case, also the
restriction to the boundary $p_{|}\: \Bd M \to \Bd N$ is simple.

\begin{definition}\label{bc-pairs/def}
Let $M$ and $N$ be compact oriented connected $n$-manifolds, and let $M_1, \dots, M_k
\subset M$ and $N_1, \dots, N_k \subset N$ be compact oriented locally flat PL
submanifolds embedded in $M$ and $N$, respectively. By a \emph{$d$-fold branched covering}
$p \: (M; M_1, \dots, M_k) \to (N; N_1, \dots, N_k)$ we mean a $d$-fold branched covering
$p \: M \to N$ whose branch set is transversal to all the submanifolds $N_i$ and such that
$p(M_i) = N_i$ and $p_i=p_{|M_i}\: M_i \to N_i$ preserves the orientation for every
$i=1,\dots,k$.
\end{definition}

We note that, if $p$ is a (simple) $d$-fold branched covering as in the definition, then
each restriction $p_i \: M_i \to N_i$ is a (simple) $d_i$-fold branched covering for some
$d_i \leq d$.

Given two closed oriented locally flat PL surfaces $F_1,F_2 \subset M$ in the closed
oriented PL 4-manifold $M$, we will denote by $F_1 \cdot F_2$ their algebraic
intersection, that is the number $\beta_M([F_1],[F_2]) \in \Z$.

We also need the following technical definition. First, we remind that a properly embedded
locally flat PL surface $S \subset B^4$ is said to be \emph{ribbon} if the distance
function from the origin restricted to $S$, has no local maxima in $\Int S$. In
particular, a push in of a PL surface embedded in $S^3 \subset B^4$ is ribbon.

\begin{definition}\label{ribbon-ext/def}
A simple branched covering $p \: M \to S^3$ is said to be \emph{ribbon fillable} if it can
be extended to a simple branched covering $q \: W \to B^4$ whose branch set $B_q \subset
B^4$ is a ribbon surface (which immediately implies that $M = \Bd W$, $B_p = \Bd B_q
\subset S^3$ is a link, and $d(p) = d(q)$). For the sake of convenience, we also call
ribbon fillable any simple branched cover $p \: M \to S^3_1 \cup \dots \cup S^3_k$ that is
a disjoint union of ribbon fillable coverings.
\end{definition}

This definition is relevant in light of the following theorem, which summarises a
classical result for 4-dimensional branched coverings due to Montesinos \cite{Mo78} (see
also \cite{BP05,BP12} for an explicit direct construction, starting from a Kirby diagram),
and an application of it to 3-manifolds obtained by taking into account the
Lickorish-Wallace theorem \cite{Li62,W60}.

\begin{theorem}\label{MLW/thm}
Any compact connected oriented $4$-dimensional $2$-handlebody $W$ is a simple $3$-fold
covering of $B^4$ branched over a ribbon surface in $B^4$. Hence, every closed connected
oriented $3$-manifold is a ribbon fillable $3$-fold branched covering of $S^3$.
\end{theorem}

The degree of a branched covering of a sphere or a ball can be arbitrarily increased by
iterating the operation of stabilisation, according to the following definition.

\begin{definition}\label{stab/def}
For any $d$-fold branched covering $p \: M \to N$, where $N \cong S^n$ or $N \cong B^n$,
the \emph{covering stabilisation} of $p$ is the $(d+1)$-fold branched covering $M \to N$
obtained from $p$ by adding to the branch set $B_p$ a separate trivial $(n-2)$-sphere in
$S^n$ or proper $(n-2)$-ball in $B^n$, respectively, with monodromy $(d\ d\,{+}\,1)$ for a
meridian of it. By $k$ subsequent applications of this operation, we get a $(d+k)$-fold
branched covering $p' \: M \to N$, which we call the \emph{$k$-fold stabilisation of $p$}.
By construction, $p'$ turns out to be a simple branched covering if $p$ is simple.
Moreover, if $N \cong S^3$ and $p$ is ribbon fillable, then also $p'$ is ribbon fillable.
\end{definition}

The proofs of our results depend on the following theorem, which was established in
\cite{PZ16}.

\begin{theorem}\label{bc-ext/thm}
Let $M$ be a compact connected oriented PL $4$-manifold whose boundary has $k$ connected
components, and let $B^4_1, \dots, B^4_k \subset S^4$ be a collection of pairwise disjoint
PL 4-balls bounded by the 3-spheres $S^3_1, \dots, S^3_k \subset S^4$, respectively. Any
$d$-fold ribbon fillable simple branched covering $p \: \Bd M \to S^3_1 \cup \dots \cup
S^3_k$ of degree $d \geq 4$, extends to a simple $d$-fold covering $q \: M \to S^4 -
\Int(B^4_1 \cup \dots \cup B^4_k)$ whose branch set $B_q$ is a locally flat
self-transversal PL surface properly immersed $($embedded for $d \geq 5)$ in $S^4 -
\Int(B^4_1 \cup \dots \cup B^4_k)$.
\end{theorem}

\section{Branched coverings of disc bundles and their plumbings}\label{plumb/sec}

Given a closed connected oriented surface $F$, we denote by $\xi_{F,e} \: D_{F,e} \to F$
the oriented disc bundle over $F$ with Euler number $e \in \Z$. By abusing notation, we
also write $F \subset \Int D_{F,e}$ to indicate (the properly embedded oriented surface
image of) a PL section $F \to \Int D_{F,e}$\break of $\xi_{F,e}$.

\begin{proposition}\label{bundle/thm}
If $p \: F \to G$ is a $($simple\/$)$ branched covering of degree $d \geq 1$ between
closed connected oriented surfaces, then the pullback $p^*(\xi_{G,e})$ is bundle
isomorphic to $\xi_{F,de}$ for every $e \in \Z$. Moreover, for any PL sections $F \subset
\Int D_{F,de}$ and $G \subset \Int D_{G,e}$, $p$ lifts to a fiber-preserving
$($simple\/$)$ branched covering $\widetilde p \: (D_{F,de};F) \to (D_{G,e};G)$ having the
same degree $d$ and branch set the disjoint union of fiber discs $B_{\widetilde p} =
\xi_{G,e}^{-1}(B_p)$.
\end{proposition}

\begin{proof}
To prove the bundle isomorphism $p^*(\xi_{G,e}) \cong \xi_{F,de}$, it is enough to
consider two sections $G',G'' \subset \Int D_{G,e}$ of $\xi_{G,e}$ that intersect each
other transversally away from $\xi_{G,e}^{-1}(B_p)$, and observe that the pullback
sections $F',F''$ of $p^*(\xi_{G,e})$ satisfy $F' \cdot F''=d (G' \cdot G'') = de$.

Up to the above isomorphism, we obtain a lifting $\widetilde p \: D_{F,de} \to D_{G,e}$
associated to the pullback, which is a (simple) branched covering with branch set
$\xi_{G,e}^{-1}(B_p)$, due to the local product structure of the bundles. Moreover, given
any two sections $F$ and $G$ as in second part of the statement, we can attain $\widetilde
p(F) = G$ by composing $\widetilde p$ with an arbitrary bundle automorphism of
$\xi_{F,de}$ that sends $F$ to the pullback of $G$.
\end{proof}

\begin{proposition}\label{bundleS/thm}
For any connected simple branched covering $p\:F \to S^2$ of degree $d \geq 1$, the simple
branched covering $\widetilde p \: (D_{F,\pm d};F) \to (D_{S^2,\pm 1};S^2)$ given by the
previous proposition, restricts to a ribbon fillable branched covering $\widetilde
p_{|\Bd}\: \Bd D_{F,\pm d} \to \Bd D_{S^2,\pm 1} \cong S^3$.
\end{proposition}

\begin{proof}
By the L\"uroth-Clebsch theorem (see Berstein and Edmonds \cite{BE79}, or Bauer and
Catanese \cite{BC97} for a different approach), simple branched coverings from a closed
connected oriented genus $g$ surface to $S^2$ are classified by the degree. Therefore, up
to covering equivalence we can assume that $p$ is the $(d-2)$-fold stabilisation of the
hyperelliptic 2-fold covering $F \to S^2$.

Let $n = g(F)+d-1$. Then $B_p$ consists of $2n$ points $a_1,a'_1, \dots , a_n, a'_n$
having monodromies $(1\,2),(1\,2),\dots,(1\,2), (1\,2), (2\,3), (2\,3), \dots, (d{-}1\,d),
(d{-}1\,d)$, with respect to a suitable Hurwitz system. Thus, the branch set
$B_{\widetilde p}$ consists of $2n$ discs with those monodromies.

Since the restriction to the boundary of the bundle $\xi_{S^2,\pm1}$ is a Hopf fibration,
the branch set of $\widetilde p_{|\Bd }$, which is the boundary of $B_{\widetilde p}$,
consists of $2n$ Hopf fibers $C_1,C_1', \dots, C_n,C_n'$, such that $C_i$ and $C_i'$ have
the same monodromy.

Now, there exist $n$ pairwise disjoint properly embedded ribbon annuli $R_1, \dots, R_n
\subset B^4$, such that $\Bd R_i = C_i \cup C_i'$. Indeed, these can be obtained as the
push in of the preimages by $\xi_{S^2,\pm1}$ of $n$ pairwise disjoint arcs $A_1,\dots, A_n
\subset S^2$, such that each arc $A_i$ joins $a_i$ and $a_i'$. By choosing these arcs so
that they meet the Hurwitz system only at their end points, the monodromy of $\widetilde
p_{|\Bd }$ can be extended over $B^4 - (R_1 \cup \dots \cup R_n)$, yielding a ribbon
filling of that covering.
\end{proof}

For any disc bundles $\xi_{F_1,e_1}:D_{F_1,e_1} \to F_1$ and $\xi_{F_2,e_2}:D_{F_2,e_2}
\to F_2$ over closed connected oriented surfaces $F_1$ and $F_2$, we can form the
\emph{positive $n$-fold plumbing} $X_n(\xi_{F_1,e_1}, \xi_{F_2,e_2})$ of $D_{F_1,e_1}$ and
$D_{F_2,e_2}$ as follows. We choose two families of pairwise disjoint discs $U_1, \dots,
U_n \subset F_1$ and $V_1,\dots, V_n \subset F_2$, together with local trivializations
$\xi^{-1}_{F_1,e_1}(U_i) \cong B^2 \times B^2$ and $\xi^{-1}_{F_2,e_2}(V_i) \cong B^2
\times B^2$ of the bundles, for $i = 1,\dots, n$. Then, we define the oriented PL
4-manifold
\begin{equation}\label{X/eqn}
X_n(\xi_{F_1,e_1}, \xi_{F_2,e_2}) = D_{F_1,e_1} \cup_{\phi_1\cup\dots \cup\phi_n} 
D_{F_2,e_2},
\end{equation}
where the gluing homeomorphisms $\phi_i\: \xi^{-1}_{F_1,e_1}(U_i) \to
\xi^{-1}_{F_2,e_2}(V_i)$ are assumed to interchange the base and the fiber up to those
local trivializations. We can consider $D_{F_1,e_1}$ and $D_{F_2,e_2}$ as subspaces of
$X_n(\xi_{F_1,e_1}, \xi_{F_2,e_2})$, and we call each connected component
$\xi^{-1}_{F_1,e_1}(U_i) \cong_{\phi_i} \xi^{-1}_{F_2,e_2}(V_i)$ of $D_{F_1,e_1} \cap
D_{F_2,e_2}$ a \emph{plumbing region}.

Given two PL sections $F_1 \subset \Int D_{F_1,e_1}$ and $F_2 \subset \Int D_{F_2,e_2}$ of
$\xi_{F_1,e_1}$ and $\xi_{F_2,e_2}$, respectively, we can choose the above trivializations
in such a way that all the intersections
\pagebreak
$F_1 \cap \xi^{-1}_{F_1,e_1}(U_i)$ and $F_2 \cap \xi^{-1}_{F_2,e_2}(V_i)$ correspond to $B^2 \times 
\{0\} \subset B^2 \times B^2$. In this way, we can consider
\begin{equation}
F_1,F_2 \subset \Int X_n(\xi_{F_1,e_1}, \xi_{F_2,e_2})
\end{equation}
as properly embedded oriented PL surfaces which intersect transversally and positively at
$n$ points, in such way that $X_n(\xi_{F_1,e_1}, \xi_{F_2,e_2})$ can be thought as a
regular neighborhood of $F_1 \cup F_2$.

\begin{remark}\label{independent/rmk}
The triple $(X_n(\xi_{F_1,e_1}, \xi_{F_2,e_2}); F_1, F_2)$ does not depend, up to PL
homeomorphisms, on the choices involved in the construction.
\end{remark}

\begin{proposition}\label{plumbing/thm}
For any $($simple\/$)$ $d$-fold branched coverings $p_1 \: F_1 \to G_1$ and $p_2 \: F_2
\to G_2$ between closed connected oriented surfaces, any disc bundles\, $\xi_{F_i,d e_i}$
and\, $\xi_{G_i,e_i}$, and any PL sections $F_i \subset \xi_{F_i,d e_i}$ and $G_i \subset
\xi_{G_i,e_i}$, for $i =1,2$, there exists a $($simple\/$)$ $d$-fold branched covering
$$\widetilde p: (X_d(\xi_{F_1,d e_1}, \xi_{F_2,d e_2}); F_1,F_2) \to (X_1(\xi_{G_1,e_1},
\xi_{G_2,e_2}); G_1,G_2).$$
In addition, $\widetilde p$ is fiber-preserving away from the plumbing regions and sends
each plumbing region upstairs homeomorphically to the plumbing region downstairs, and the
branch set $B_{\widetilde p}$ is a disjoint union of fiber discs, coinciding with
$B_{\widetilde p_1} \cup B_{\widetilde p_2}$.
\end{proposition}

\begin{proof}
Proposition \ref{bundle/thm} yields $d$-fold fiber-preserving branched coverings
$\widetilde p_1 \: D_{F_1,d e_1} \to D_{G_1,e_1}$ and $\widetilde p_2 \: D_{F_2,de_2} \to
D_{G_2,e_2}$.

Let us consider the two discs $U \subset G_1$ and $V \subset G_2$ that determine the
plumbing region of $X_1(\xi_{G_1,e_1}, \xi_{G_2,e_2})$ as $\xi_{G_1,e_1}^{-1}(U)
\cong_\phi \xi_{G_2,e_2}^{-1}(V)$, where $\phi$ is the gluing homeomorphism. By Remark
\ref{independent/rmk}, we can assume that $U \cap B_{p_1} = \emptyset$ and $V \cap B_{p_2}
= \emptyset$.

It follows that $p_1^{-1}(U)$ is a disjoint union of $d$ discs $U_1,\dots, U_d \subset
F_1$, and similarly $p_2^{-1}(V)$ is a disjoint union of $d$ discs $V_1,\dots, V_d \subset
F_2$. Taking into account that $\widetilde p_1$ and $\widetilde p_2$ are fiber-preserving,
by Remark \ref{independent/rmk} again, we can assume that the plumbing regions of
$X_d(\xi_{F_1,d e_1}, \xi_{F_2,d e_2})$ are $\xi_{F_1,de_1}^{-1}(U_i) \cong_{\phi_i}
\xi_{F_2,de_2}^{-1}(V_i)$, and that the gluing homeomorphisms $\phi_i\:
\xi_{F_1,de_1}^{-1}(U_i) \to \xi_{F_2,de_2}^{-1}(V_i)$ are determined by the equations
$$\widetilde p_2 \circ \phi_i = \phi \circ \widetilde p_1,$$
for $i=1,\dots,d$. Therefore, the maps $\widetilde p_1$ and $\widetilde p_2$ can be glued
together to give a map 
$$\widetilde p\: X_d(\xi_{F_1,d e_1}, \xi_{F_2,d e_2}) \to
X_1(\xi_{G_1,e_1}, \xi_{G_2,e_2}),$$ 
which in turn is a branched covering since the gluing
is by homeomorphisms, and it is fiber-preserving away from the plumbing regions because so
are $\widetilde p_1$ and $\widetilde p_2$. Thus, $B_{\widetilde p} = B_{\widetilde p_1}
\cup B_{\widetilde p_2}$ is a disjoint union of fiber discs. Finally, the equalities
$\widetilde p(F_1) = \widetilde p_1(F_1)= G_1$ and $\widetilde p(F_2) = \widetilde
p_2(F_2)= G_2$, and the fact that $\widetilde p$ sends each plumbing region upstairs
homeomorphically to the plumbing region downstairs, are obvious by the construction.
\end{proof}

\begin{proposition}\label{bundleSxS/thm}
Given any connected simple branched coverings $p_1\:F_1 \to S^2$ and $p_2\:F_2 \to S^2$ of
degree $d \geq 1$ and any integer $e \in \Z$, we have that $\Bd X_1(\xi_{S^2,e},
\xi_{S^2,0})\cong S^3$ and the simple branched covering
$$\widetilde p\: X_d(\xi_{F_1,d e}, \xi_{F_2, 0}) \to X_1(\xi_{S^2,e}, \xi_{S^2,0}),$$
of the previous proposition, restricts to a ribbon fillable branched covering
$$\widetilde p_{|\Bd}\: \Bd X_d(\xi_{F_1,d e}, \xi_{F_2, 0}) \to \Bd X_1(\xi_{S^2,e}, 
\xi_{S^2,0})\cong S^3.$$
\end{proposition}

\begin{proof}
The manifold $X_1(\xi_{S^2,e}, \xi_{S^2,0})$ admits a handlebody decomposition with two
2-handles attached to $B^4$ along the components of the Hopf link, one with framing $e$ to
give $D_{S^2,e}$ and the other with framing $0$ to give $D_{S^2,0}$. In the corresponding
Kirby diagram of the boundary, the 0-framed component of the framed link can be cancelled
with the $e$-framed one to give a PL homeomorphism $\Bd X_1(\xi_{S^2,e}, \xi_{S^2,0})\cong
S^3$.

The branch set $B_{\widetilde p}$ coincides with $B_{\widetilde p_1} \cup B_{\widetilde
p_2}$ by the previous proposition, and in the above handlebody decomposition is given by
$2(g(F_1) + d - 1)$ discs parallel to the co-core of the 2-handle with framing $e$ and
$2(g(F_2) + d - 1)$ discs parallel to the co-core of the other 2-handle. The discs of each
family come in pairs with equal monodromies, as in the proof of Proposition
\ref{bundleS/thm}.

By looking at the boundary, we get the left side of Figure \ref{ribbons/fig}, which
depicts $B_{\widetilde p_1| \Bd}$ and $B_{\widetilde p_2| \Bd}$ as two families of circles
linked with the corresponding framed unknots. Up to the PL homeomorphism $\Bd
X_1(\xi_{S^2,e}, \xi_{S^2,0}) \cong S^3$, we get the boundary link in the right side of
Figure \ref{ribbons/fig}. To see this, we first slide the circles corresponding to
$B_{\widetilde p_2| \Bd}$, over the unknot with framing $e$, making them unlinked with the
one with framing 0. Subsequently, we slide all the branch circles over the 0-framed unknot
to separate them from the framed link, which can be now cancelled, realising the surgery
that yields the PL homeomorphism with $S^3$.

\begin{figure}[htb]
\centering
\includegraphics{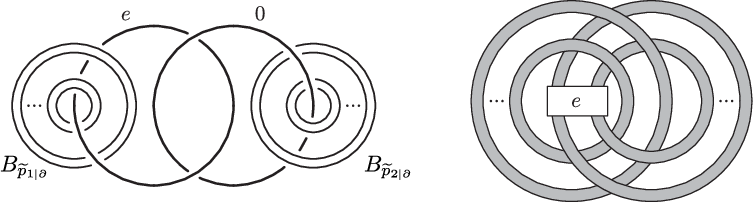}
\caption{Ribbon fillability of $\widetilde p_{|\Bd}$.}
\label{ribbons/fig}
\end{figure}

At this point, the ribbon fillability follows as in the last part of the proof of
Proposition \ref{bundleS/thm}, by extending the monodromy over the complement in $B^4$ of
a family of $g(F_1) + g(F_2) + 2d - 2$ bands, which are the push in of the bands showed in
the right side of Figure \ref{ribbons/fig}.
\end{proof}

\section{Branched coverings constructions for submanifolds}\label{BC-sub/sec}

We are ready to state and prove our results for branched coverings relative to certain
submanifolds, as we mentioned in the Introduction.

\begin{theorem}\label{bc2-cp2/thm}
Let $M$ be a closed connected oriented PL $4$-manifold and $F \subset M$ be a closed
connected oriented locally flat PL surface. If $d = |F \cdot F| \geq 4$, then there exists
a simple $d$-fold branched covering:
\begin{itemize}
\item[\(a\)] $p \: (M;F) \to (\CP^2;\CP^1)$ if $F \cdot F$ is positive;
\item[\(b\)] $p \: (M;F) \to (\CPbar^2;\CP^1)$ if $F \cdot F$ is negative.
\end{itemize}
In both cases, $F = p^{-1}(\CP^1)$, and $B_p$ is a closed locally flat PL surface
self-transver\-sally immersed $($embedded for $d \geq 5)$ in $\CP^2$ or $\CPbar^2$.
\end{theorem}

\begin{proof}
Case \(b\) immediately follows from case \(a\) by reversing the orientation of $M$.
So, it suffices to prove case \(a\), supposing $d = F \cdot F \geq 4$.

Let $T_F \subset M$ be a tubular neighborhood of $F$ in $M$, and $T_{\CP^1} \subset \CP^2$
be a tubular neigh\-bor\-hood of $\CP^1$ in $\CP^2$. Then, given any simple $d$-fold
branched covering $f\: F \to S^2$ and taking into account the PL homeomorphisms $T_F \cong
D_{F,d}$ and $T_{\CP^1} \cong D_{\CP^1,1} \cong D_{S^2,1}$, we can apply Proposition
\ref{bundle/thm} to obtain a simple $d$-fold branched covering $t \: (T_F, F) \to
(T_{\CP^1}, \CP^1)$. Moreover, the restriction $t_{|\Bd} \: \Bd T_F \to \Bd T_{\CP^1}$ is
ribbon fillable by Proposition \ref{bundleS/thm}.

We set $W = \Cl(M - T_F)$ and $Y = \Cl(\CP^2 - T_{CP^1}) \cong B^4$. Then, Theorem
\ref{bc-ext/thm} allows us to extend $t_{|\Bd}$ to a simple $d$-fold covering $q\:W \to Y$
branched over a self-transversally immersed (embedded for $d \geq 5$) surface.

Finally, we can define the desired covering $p$ as the union of the coverings $t$ and $q$,
which share the same restriction to the boundary. Namely, $p = t \cup_\Bd q \: M = T_F
\cup_\Bd W \to \CP^2 = T_{\CP^1} \cup_\Bd Y$.
\end{proof}

\begin{theorem}\label{bc2-s2xs2/thm}
Let $M$ be a closed connected oriented PL $4$-manifold and $F_1,F_2 \subset M$ be two
closed connected oriented locally flat PL surfaces transversal to each other, whose all
intersection points are positive. If $F_1 \cdot F_1 = n\,d$, $F_1 \cdot F_2 = d$ and $F_2
\cdot F_2 = 0$ for some integers $n$ and $d \geq 4$, then there exists a simple $d$-fold
branched covering:
\begin{itemize} 
\item[\(a\)]$p \: (M; F_1,F_2) \to (S^2 \times S^2; S^2_1,S^2_2)$, with $S^2_1$ and
$S^2_2$ respectively a section with self-in\-ter\-sec\-tion $n$ and a fiber of the trivial
bundle $S^2 \times S^2 \to S^2$, if $n$ is even;
\item[\(b\)]$p \: (M; F_1,F_2) \to (S^2 \simtimes S^2; S^2_1,S^2_2)$, with $S^2_1$ and
$S^2_2$ respectively a section with self-in\-ter\-sec\-tion $n$ and a fiber of the twisted
bundle $S^2 \simtimes S^2 \to S^2$, if $n$ is odd.
\end{itemize}
In both cases, $F_i = p^{-1}(S^2_i)$, and $B_p$ is a closed locally flat PL surface
self-transver\-sally im\-mersed $($embedded for $d \geq 5)$ in $S^2 \times S^2$ or $S^2
\simtimes S^2$.
\end{theorem}

We observe that a section as specified in the above statement exists for every integer
$n$. In fact, given two copies of the trivial bundle $B^2 \times S^2 \to B^2$, we can glue
them along the boundary by the map $(\alpha,x) \mapsto (\alpha,\rho_{n\alpha}(x))$, with
$\rho_\alpha \: \R^3 \to \R^3$ the rotation of $\alpha$ radians around the third axis. In
this way, we get the trivial bundle $S^2 \times S^2 \to S^2$ or the twisted bundle $S^2
\simtimes S^2 \to S^2$, depending on the parity of $n$, with two natural sections deriving
from the two copies of $B^2 \times \{(0,0,\pm1)\}$, both having self-intersection $n$.

\begin{proof}
For the sake of convenience, we denote by $\xi \: E \to S^2$ the trivial bundle $S^2
\times S^2 \to S^2$ or the twisted bundle $S^2 \simtimes S^2 \to S^2$, depending on
whether $n$ is even or odd.

We can choose tubular neighborhoods $T_{F_1}$ of $F_1$ and $T_{F_2}$ of $F_2$, in such a
way that their union $T_{F_1} \cup T_{F_2}$ is a regular neighborhood of $F_1 \cup F_2$ in
$M$. It follows that there is a PL homeomorphism
$$(T_{F_1} \cup T_{F_2}; F_1, F_2) \cong (X_d(\xi_{F_1, d n}, \xi_{F_2,0}); F_1, F_2),$$ 
where $X_d(\xi_{F_1, d n}, \xi_{F_2,0})$ is the $d$-fold plumbing defined in Section 
\ref{plumb/sec}.

Similarly, we can choose tubular neighborhoods $T_{S^2_1}$ of $S^2_1$ and $T_{S^2_2}$ of
$S^2_2$, in such a way that their union $T_{S^2_1} \cup T_{S^2_2}$ is a regular
neighborhood of $S^2_1 \cup S^2_2$ in $E$. As above, there is a PL homeomorphism
$$(T_{S^2_1} \cup T_{S^2_2}; S^2_1, S^2_2) \cong (X_1(\xi_{S^2_1, n}, \xi_{S^2_2,0}); 
S^2_1, S^2_2).$$

Now, let $f_1 \: F_1 \to S^2_1$ and $f_2 \: F_2 \to S^2_2$ be simple $d$-fold branched
coverings. By Proposition \ref{plumbing/thm}, we get a simple $d$-fold branched covering
$$t \: (T_{F_1} \cup T_{F_2}; F_1, F_2) \to (T_{S^2_1} \cup T_{S^2_2}; S^2_1, S^2_2),$$ 
whose restriction $t_{|\Bd} \: \Bd (T_{F_1} \cup T_{F_2}) \to \Bd (T_{S^2_1} \cup
T_{S^2_2})$ is ribbon fillable by Proposition \ref{bundleSxS/thm}.

Looking at the complement of those tubular neighborhoods, we put $W = \Cl(M - (T_{F_1}
\cup T_{F_2}))$ and $Y = \Cl(E - (T_{S^2_1} \cup T_{S^2_2})) \cong B^4$. Then, we can use
Theorem \ref{bc-ext/thm} for extending $t_{|\Bd}$ to a simple $d$-fold covering $q \: W
\to Y$ branched over a self-transversally immersed (embedded for $d \geq 5$) PL surface,
and conclude the proof by putting $p = t \cup_\Bd q \: M = (T_{F_1} \cup T_{F_2}) \cup_\Bd
W \to E = (T_{S^2_1} \cup T_{S^2_2}) \cup_\Bd Y$.
\end{proof}

\begin{theorem}\label{bc3/thm}
Let $M$ be a closed connected oriented PL $4$-manifold and $N \subset M$ be a closed
connected oriented $($locally flat\/$)$ PL $3$-manifold. For any $d \geq 4$ there exists a
simple $d$-fold branched covering:
\begin{itemize}
\item[\(a\)] $p \: (M;N) \to (S^4;S^3)$ if $N$ disconnects $M$;
\item[\(b\)] $p \: (M;N) \to (S^3 \times S^1;S^3 = S^3 \times \{\ast\})$ if $N$ does not 
disconnect $M$.
\end{itemize}
In both cases, $N = p^{-1}(S^3)$, and $B_p$ is a closed locally flat PL surface 
self-transver\-sally immersed $($embedded for $d \geq 5)$ in $S^4$ or $S^3 \times S^1$.
\end{theorem}

\begin{proof}
According to Theorem \ref{MLW/thm}, and up to covering stabilization, there exists a
ribbon fillable $d$-fold branched covering $c \: N \to S^3$.

If $N$ disconnects $M$, let $M_1,M_2 \subset M$ be the closures of the two connected
components of $M - N$. Then, $M_1$ and $M_2$ are two PL compact oriented 4-manifolds with
$\Bd M_1 = \Bd M_2 = N$, such that $M = M_1 \cup M_2$. By Theorem \ref{bc-ext/thm}, the
branched covering $c$ extends to two simple $d$-fold branched coverings $p_1 \: M_1 \to
S^4_-$ and $p_2 \: M_2 \to S^4_+$, both branched over a locally flat PL surface
self-transversally immersed (embedded if $d \geq 5$) in the base space, where $S^n_\pm
\subset S^n$ are the two hemispheres bounded by $S^{n-1} \subset S^n$. Therefore, we can
put $p = p_1 \cup p_2 \: M \to S^4$.

In the case where $N$ does not disconnect $M$, we consider the decomposition $M = M_1 \cup
M_2$, with $M_1$ a collar of $N$ in $M$, and $M_2 = \Cl(M - M_1)$. The simple $d$-fold
covering $p_1 = c \times \id_{S^1_-} \: M_1 \cong N \times S^1_- \to S^3 \times S^1_-$ is
branched over the locally flat PL surface $B_c \times S^1_-$, which is properly embedded
in $S^3 \times S^1_-$. The restriction of $p_1$ to the boundary is a ribbon fillable
$d$-fold branched covering $\Bd M_1 = \Bd M_2 \to \Bd (S^3 \times S^1_-) = \Bd(S^3 \times
S^1_+)$, which by Theorem \ref{bc-ext/thm} admits a simple $d$-fold extension $p_2 \: M_2
\to S^3 \times S^1_+$ branched over a locally flat PL surface self-transversally immersed
(embedded if $d \geq 5$) in $S^3 \times S^1_+$. So, also in this case we can conclude by
putting $p = p_1 \cup p_2 \: M \to S^3 \times S^1$.
\end{proof}

Our last result of this section is not related to the main theorem. Still, we include it
for the sake of completeness, since it provides a representation of surfaces in
4-manifolds as branched covering of trivial 2-spheres in $S^4$ (cf. \cite{MP98} for links
in 3-manifolds).

\begin{theorem}\label{bc2/thm}
Let $M$ be a closed connected oriented PL $4$-manifold and $F \subset M$ be a closed
oriented locally flat PL surface with $k$ connected components $F_1,\dots,F_k$, such that
$F_i \cdot F_i = 0$ for every $i = 1,\dots,k$ $($that is, the normal bundle $\nu_F$ is
trivial\/$)$. Then, for any $d \geq 4$ there is a simple $d$-fold branched covering $p \:
(M; F) \to (S^4; T_k)$, with $T_k \subset S^4$ the trivial $2$-link with $k$ spherical
components and $B_p \subset S^4$ a closed locally flat PL surface self-transversally
immersed $($embedded for $d \geq 5)$ in $S^4$, which is transversal to $T_k$. Moreover,
$p$ can be chosen in such a way that each restriction $p_{|F_i} \: F_i \to p(F_i) \cong
S^2$ is equivalent to any given simple branched covering of degree $d_i \leq d - 2$. In
particular, if $F$ is consists of $2$-spheres, we can assume $B_p \cap T_k = \emptyset$,
hence $p$ is the trivial $d$-fold covering over $T_k$.
\end{theorem}

We note that any closed oriented locally flat PL surface $F \subset S^4$ admits a branched
covering representation as in the theorem.

For the proof of Theorem \ref{bc2/thm} we need two lemmas.

\begin{lemma}\label{ribbon/thm}
Let $C \subset B^3$ be a properly embedded $($not necessarily connected\/$)$ compact
curve. Then, the surface $F = C \times B^1 \subset B^3 \times B^1 \cong B^4$ is ribbon.
\end{lemma}

\begin{proof}
Up to ambient isotopy, we can assume that the origin $0 \in B^3$ does not belong to $C$
and that the image $D = \pi_0(C) \subset S^2$ of $C$ under the radial projection $\pi_0 \:
B^3 - \{0\} \to S^2$ from $0$ forms only transversal double points (it gives a diagram of
$C$).\break Let $\pi_{(0,0)} \: (B^3 \times B^1) -\{(0,0)\} \to \Bd(B^3 \times B^1) = (S^2
\times B^1) \cup (B^3 \times S^0) \cong S^3$ the radial\break projection from the origin $(0,0)
\in B^3 \times B^1$. Then, for each $x \in C$
\pagebreak the image under $\pi_{(0,0)}$ of the segment $\{x\} 
\times B^1$ is given by $\pi_{(0,0)}(\{x\}\times B^1) = (\pi_0(\{x\}) \times B^1) \cup ([x,\pi_0(x)] 
\times S^0) \subset \Bd(B^3 \times B^1)$, where $[x,\pi_0(x)] \subset B^3$ denotes the 
segment spanned by $x$ and $\pi_0(x)$. It follows that the image $\pi_{(0,0)}(F) \subset 
\Bd(B^3 \times B^1)$ forms only ribbon intersections, consisting of a single double arc 
for each double point of $D$. Hence, $F$ is a ribbon surface.
\end{proof}

\begin{remark}\label{ribbon/rem}
In the smooth category, one could argue that the surface $F$ can be realized in $B^4$ as a
ruled surface, not passing through the origin. Then, the distance from the origin
restricts to a function on $F$ without local maxima in $\Int F$, which implies that $F$ is
ribbon.
\end{remark}

\begin{lemma}\label{bc3bd/thm}
Let $N_1,\dots,N_k \subset M$ be pairwise disjoint compact oriented $($locally flat\/$)$
PL $3$-manifolds with non-empty boundary, and let $B^3_1,\dots,B^3_k \subset S^4$ be
pairwise disjoint PL 3-balls. For every $i=1, \dots, k$, let $c_i \: N_i \to B^3_i$ be a
simple $d$-fold branched covering, with $B_{c_i} \subset B^3_i$ a properly embedded
compact curve and $d \geq 4$. Then, $c = c_1 \cup \dots \cup c_k$ extends to a simple
$d$-fold branched covering $p \: M \to S^4$ with $B_p$ a locally flat PL surface
self-transversally immersed $($embedded for $d \geq 5)$ in $S^4$.
\end{lemma}

\begin{proof}
We consider pairwise disjoint collars $C_i = C(N_i) \subset M$ of the 3-manifolds $N_i$ in
$M$ and pairwise disjoint collars $D_i = C(B^3_i) \subset S^4$ of the 3-balls $B^3_i$ in
$S^4$. Then, we have $C_i \cong N_i \times B^1$ with $N_i$ canonically identified to $N_i
\times \{0\}$, and $D_i \cong B^3_i \times B^1$ with $B^3_i$ canonically identified to
$B^3_i \times \{0\}$. Up to these identifications and assuming all the collars positively
oriented, the branched coverings $c_i$ extend to simple $d$-fold coverings $c'_i = c_i
\times \id_{B^1} \: C_i \to D_i$. By Lemma \ref{ribbon/thm}, each branch set $B_{c'_i}$ is
a ribbon surface in $D_i \cong B^4$. Now, we consider the simple $d$-fold branched
covering $p_1 = \cup_i c'_i \: {\cup_i} C_i \to \cup_i D_i$, and put $X = \Cl(M - \cup_i
C_i)$ and $Y = \Cl(S^4 - \cup_i D_i)$. The restriction to the boundary of $p_1$ gives a
simple $d$-fold branched covering $p_{1|\Bd} \: \Bd X \to \Bd Y$, which is ribbon fillable
by construction. Therefore, Theorem \ref{bc-ext/thm} allows us to extend $p_{1|\Bd}$ to a
simple $d$-fold branched covering $p_2 \: X \to Y$ with $B_{p_2}$ a locally flat PL
surface self-transversally immersed (embedded for $d \geq 5$) in $Y$. Thus, we can
conclude the proof by putting $p = p_1 \cup_\Bd p_2$.
\end{proof}

\begin{proof}[Theorem \ref{bc2/thm}]
Since the normal bundle $\nu_F$ is trivial, for every $i = 1, \dots, k$ we can find a
3-dimensional locally flat PL ribbon $N_i \cong F_i \times [0,1]$ in $M$ such that $\Bd
N_i = F_i \cup F'_i$, with $F'_i \subset M$ a ``parallel'' copy of $F_i$ oriented in the
opposite way. We assume the $N_i$'s to be pairwise disjoint. Let $N_i' \subset M$ be the
3-manifold obtained by removing the interiors of $d - d_i - 2$ disjoint PL 3-balls from
$\Int N_i$.

Each surface $\Bd N_i'$ admits a $d$-fold simple branched covering $f_i \: \Bd N_i' \to
S^2$, where $F_i$ consists of the sheets $1$ to $d_i$, $F'_i$ consists of the sheets
$d_i+1$ and $d_i+2$, while the boundaries of the removed 3-balls consists of the remaining
$d - d_i - 2 \geq 0$ sheets trivially covering $S^2$. By Corollary 6.3 in \cite{BE79},
this can be extended to a $d$-fold simple branched covering $c_i \: N_i' \to B^3$. After
having identified the base spaces of such coverings with a family of disjoint PL 3-balls
$B^3_1, \dots, B^3_k \subset S^4$, we can apply Lemma \ref{bc3bd/thm} to get a simple
covering $p \: (M; N_1', \dots, N'_k) \to (S^4; B^3_1,\dots, B^3_k)$ of degree $d$,
branched over a locally flat PL surface self-transversally immersed (embedded for $d \geq
5$) in $S^4$. Then, $p$ is the desired branched covering, since $p(F) = \Bd(\cup_i\,
B^3_i)$ is a trivial link of $k$ spheres. Moreover, by the L\"uroth-Clebsch classification
of simple branched coverings of $S^2$ (see \cite{BE79} or \cite{BC97}), the restrictions
$p_{|F_i}$ can be arbitrarily chosen, up to isotopy, with the given degrees $d_i$.

If $F_i \cong S^2$ for every $i$, we set $d_i = 1$ and at the beginning of the proof we
remove the interiors of $d-2$ balls from $N_i$ (instead of $d-3$) so that $N_i'$ has $d$
boundary components, all homeomorphic to a sphere. Then, by following the same argument,
we obtain the desired simple branched covering $p\: M \to S^4$ such that $T_k \cap B_p =
\emptyset$.
\end{proof}

\section{The proofs of the main theorems}

In this section we prove the Theorems \ref{main/thm} and \ref{main-gen/thm} stated in the
Introduction. For that we need some algebraic properties of the intersection forms of PL
4-manifolds, which are stated in the next lemmas.

\begin{lemma}\label{class1/thm}
Let $M$ be a closed connected oriented PL $4$-manifold. If $b_2^+(M) \geq 1$, there exists
a class $\phi \in H_2(M)/\Tor H_2(M)$ such that $\beta_M(\phi, \phi) = k$ for each of the
followings
$$k = \cases{
\, 4 & \text{in any case}\cr
\, 6 & \text{if $\beta_M$ is even}\cr
\, 9 & \text{if $\beta_M$ is odd}\cr
\, 5 & \text{if $\beta_M$ is odd and $b_2(M)\geq2$}.}$$
If in addition $b_2^-(M) \geq 1$, there exist two classes $\phi_1, \phi_2 \in H_2(M)/\Tor 
H_2(M)$ whose intersection matrix $\Phi = (\beta_M(\phi_i,\phi_j))$ is
$$\Phi = \pmatrix{
kn & k\,\cr
k & 0\,},$$
for every $n=0,1$ and each of the followings
$$k = \cases{
\, 4 & \text{in any case}\cr
\, 5+n & \text{if $\beta_M$ is even}\cr
\, 6-n & \text{if $\beta_M$ is odd}.}$$
\end{lemma}

\begin{proof}
We start by proving the first part, where $b_2^+(M)\geq 1$. If $\beta_M$ is odd, then it
is diagonalizable. This follows by a theorem of Donaldson for definite intersection forms
of closed oriented PL 4-manifolds \cite{Do87}, and by the Serre classification theorem of
indefinite unimodular integral forms \cite{Se73,MH73}. Hence, there exists $\delta_1 \in
H_2(M)/ \Tor H_2(M)$ such that $\beta_M(\delta_1,\delta_1) = 1$, and for $b_2(M) \geq 2$
there exists also $\delta_2 \in H_2(M)/ \Tor H_2(M)$ such that $\beta_M(\delta_1,\delta_2)
= 0$ and $\beta_M(\delta_2,\delta_2) = \pm 1$.

Otherwise, if $\beta_M$ is even, then, again by Donaldson's theorem \cite{Do87}, it is
indefinite, and so it contains a hyperbolic direct summand (see \cite{MH73}, \cite{GS99}
or \cite{Ki89}). This is a sublattice having a basis $\eta_1,\eta_2 \in H_2(M)/ \Tor
H_2(M)$, such that $\beta_M(\eta_1, \eta_1) = \beta_M(\eta_2,\eta_2) = 0$ and
$\beta_M(\eta_1,\eta_2) = 1$.

In both cases, odd and even, there exists $\phi \in H_2(M)/\Tor H_2(M)$ such that
$\beta_M(\phi, \phi) = 4$, with $\phi = 2\delta_1$ for $\beta_M$ odd, and $\phi = \eta_1 +
2\eta_2$ for $\beta_M$ even. For the remaining cases, we take: $\phi = \eta_1 + 3\eta_2$,
giving $k = 6$, if $\beta_M$ is even; $\phi = 3\delta_1$, giving $k = 9$, if $\beta_M$ is
odd; $\phi = (2 - \beta_M(\delta_2, \delta_2))\,\delta_1 + 2\delta_2$, giving $k = 5$, if
$\beta_M$ is odd and $b_2(M) \geq 2$.

Next, we prove the second part, where $b_2^+(M)\geq 1$ and $b_2^-(M) \geq 1$. If $\beta_M$
is odd, then it is diagonalizable and so there exist $\delta_1, \delta_2 \in H_2(M)/\Tor
H_2(M)$ such that $\beta_M(\delta_1, \delta_1)=1$, $\beta_M(\delta_1, \delta_2)=0$, and
$\beta_M(\delta_2,\delta_2)=-1$. Then, we get: $k = 4$ and $n = 0$, for $\phi_1 = \delta_1
+ \delta_2$ and $\phi_2 = 2(\delta_1 - \delta_2)$; $k = 4$ and $n = 1$, for $\phi_1 =
2\delta_1$ and $\phi_2 = 2(\delta_1 - \delta_2)$; $k = 6$ and $n = 0$, for $\phi_1 =
\delta_1 + \delta_2$ and $\phi_2 = 3(\delta_1 - \delta_2)$; $k = 5$ and $n = 1$, for
$\phi_1 = 3\delta_1 + 2\delta_2$ and $\phi_2 = \delta_1 - \delta_2$.

If instead $\beta_M$ is even, there exists a hyperbolic pair $\eta_1,\eta_2 \in
H_2(M)/\Tor H_2(M)$, as in the analogous case of the previous part of the proof. Then, we
get: $k = 4$ and $n = 0$, for $\phi_1 = \eta_1$ and $\phi_2 = 4\eta_2$; $k = 4$ and $n =
1$, for $\phi_1 = \eta_1 + 2\eta_2$ and $\phi_2 = 4 \eta_2$; $k = 5$ and $n = 0$, for
$\phi_1 = \eta_1$ and $\phi_2 = 5\eta_2$; $k = 6$ and $n = 1$, for $\phi_1 = \eta_1 +
3\eta_2$ and $\phi_2 = 6\eta_2$.
\end{proof}

\begin{lemma}\label{class2/thm}
Let $M$ be a closed connected oriented PL 4-manifold with $b_2(M) \geq 1$. Then, for every
non-negative integers $m \leq b_2^+(M)$ and $n \leq b_2^-(M)$ there exists a sublattice
$\Lambda_{m, n}(k) \subset (H_2(M)/\Tor H_2(M), \beta_M)$ such that
$$\Lambda_{m, n}(k)
\cong \oplus_{m}\langle k\rangle \oplus_{n}\! \langle -k\rangle,$$
for each of the followings
$$k = \cases{
\, 4 & \text{in any case}\cr
\,6 & \text{if $\beta_M$ is even}\cr
\,9 & \text{if $\beta_M$ is odd}\cr
\,5 & \text{if $\beta_M$ is odd and $b_2(M) \geq 2(m+n)$}\cr}$$
where $\langle k \rangle$ is the integral rank $1$ lattice of determinant $k$.
\end{lemma}

\begin{proof}
If $\beta_M$ is odd, arguing as in the proof of Lemma \ref{class1/thm}, we have that the
lattice $(H_2(M)/\Tor H_2(M), \beta_M)$ is isomorphic to
$$\oplus_{b_2^+(M)} \langle 1\rangle \oplus_{b_2^-(M)}\! \langle -1 \rangle.$$
Thus, $\Lambda_{m,n}(4)$ and $\Lambda_{m,n}(9)$ can be obtained by taking the doubles of
some generators in the former case, and the triples in the latter. Moreover, we can obtain
$\Lambda_{m,n}(5)$ if $b_2(M) \geq 2(m+n)$ by the same argument as in the proof of Lemma
\ref{class1/thm} applied to pairs of generators.

If $\beta_M$ is even, then the lattice $(H_2(M)/\Tor H_2(M), \beta_M)$ is isomorphic to
$\oplus_a (\pm E_8) \oplus_b\! H$ for $a=|\sigma(M)|/8$ and $b = b_2^\mp(M) \geq 1$, where
$E_8$ is the symmetric rank 8 positive definite indecomposable unimodular lattice and $H$
is the unimodular hyperbolic rank 2 integral lattice.
With respect to a suitable basis, $E_8$ can be represented by the matrix
$$A_8 = \pmatrix{
\,2 & 1 & 0 & 0 & 0 & 0 & 0 &0\,\cr
\,1 & 2 & 1 & 0 & 0 & 0 & 0 & 0\,\cr
\,0 & 1 & 2 & 1 & 0 & 0 & 0 & 0\,\cr
\,0 & 0 & 1 & 2 & 1 & 0 & 0 & 0\,\cr
\,0 & 0 & 0 & 1 & 2 & 1 & 0 & 1\,\cr
\,0 & 0 & 0 & 0 & 1 & 2 & 1 & 0\,\cr
\,0 & 0 & 0 & 0 & 0 & 1 & 2 & 0\,\cr
\,0 & 0 & 0 & 0 & 1 & 0 & 0 & 2\,}.$$
In this basis, the sublattice of $E_8$ spanned by the columns $g_1, \dots, g_8$ of the 
matrix
$$G = \pmatrix{
\,0 & 0 & 0 & 0 & 0 & 0 & 0 & -2\,\cr
\,1 & 0 & 0 & 0 & 0 & 1 & 1 & 3\,\cr
\,0 & 0 & 0 & 0 & 0 & -2 & -2 & -4\,\cr
\,0 & 1 & 0 & 0 & 1 & 2 & 3 & 5\,\cr
\,0 & 0 & 0 & 0 & -2 & -2 & -4 & -6\,\cr
\,0 & 0 & 1 & 0 & 1 & 1 & 3 & 4\,\cr
\,0 & 0 & 0 & 0 & 0 & 0 & -2 & -2\,\cr
\,0 & 0 & 0 & 1 & 1 & 1 & 2 & 3\,}$$
is isomorphic to $\oplus_8 \langle 2\rangle$.
We then obtain $\oplus_8 \langle 4\rangle \subset E_8$ as the sublattice spanned by all
vectors of the form $g_{2i-1}\pm g_{2i}$ for $i\in\{ 1,2,3,4\}$.

Moreover, we obtain $\oplus_8 \langle 6\rangle \subset E_8$ as the sublattice spanned by
all vectors of the form
$$\matrix{
g_{i+1}+g_{i+2}-g_{i+3},&\quad g_{i+1}-g_{i+2}+g_{i+4},\cr
g_{i+1}+g_{i+3}-g_{i+4},&\quad g_{i+2}+g_{i+3}+g_{i+4},}$$
for $i \in \{0,4\}$.

On the other hand, we can find sublattices $\langle k\rangle \oplus \langle
-k\rangle\subset H$, for $k = 4, 6$. Therefore, the lattice $(H_2(M)/\Tor H_2(M),
\beta_M)$ with $\beta_M$ even, contains a sublattice isomorphic to
$$\oplus_{b_2^+(M)} \langle k\rangle \oplus_{b_2^-(M)}\! \langle -k\rangle$$
for $k = 4,6$, from which we get a sublattice $\Lambda_{m,n}(k)$ for $k=4,6$.
\end{proof}

\begin{lemma}\label{hyperb/thm}
Let $M$ be a closed connected oriented PL 4-manifold. Let $n$ be an integer such that
$1\leq n \leq \min(b_2^+(M), b_2^-(M))$. Then, the lattice $(H_2(M) / \Tor H_2(M),
\beta_M)$ contains a sublattice isomorphic to $\oplus_n k H$ if $\beta_M$ is even, and
sublattice isomorphic to $\oplus_n 2k H$ if $\beta_M$ is odd, for every integer $k \geq
1$, where $H$ is the unimodular hyperbolic rank 2 integral lattice. If in addition $n <
\max(b_2^+(M), b_2^-(M))$, there is a sublattice isomorphic to $\oplus_n k H$ for every
integer $k \geq 1$ also when $\beta_M$ is odd.
\end{lemma}

\begin{proof}
If $\beta_M$ is even, then there is a sublattice of $(H_2(M)/ \Tor H_2(M), \beta_M)$ which
is isomorphic to $\oplus_n H$. Then, chosen a basis of this sublattice formed by
hyperbolic pairs $\eta_1, \eta_1', \dots, \eta_n, \eta_n'$ such that $\beta_M(\eta_i,
\eta_i) = \beta_M(\eta_i', \eta_i') = 0$ and $\beta_M(\eta_i, \eta_i') = 1$ for every $i =
1,\dots, n$, we can take the sublattice spanned by all vectors of the form $\eta_i, k
\eta_i'$, for $i = 1,\dots, n$.

If instead $\beta_M$ is odd, then the intersection form is diagonalisable, hence it
contains a sublattice isomorphic to $\oplus_n(\langle 1\rangle \oplus \langle -1\rangle)$.
Let $\{\phi_1, \phi_1',\dots, \phi_n, \phi_n'\}$ be an orthogonal basis of this sublattice
such that $\beta_M(\phi_i, \phi_i) = -\beta_M(\phi_i', \phi_i') = 1$, for every
$i=1,\dots, n$. Then, it is enough to take the sublattice spanned by all vectors of the
form $\phi_i + \phi_i',\, k(\phi_i - \phi_i')$, for $i = 1,\dots,n$.

For the last part of the statement, suppose $\beta_M$ odd and $n < \max(b_2^+(M),
b_2^-(M))$. Let $a = b_2^+(M) -n$ and $b = b_2^-(M)-n$. Then, by Serre's classification
\cite{Se73,MH73}, the intersection lattice of $M$ is isomorphic to $\oplus_n H \oplus_a\!
\langle 1\rangle \oplus_b\! \langle -1 \rangle$, since this last form is indefinite, has
the same rank and signature of $M$, and it is odd because $a$ or $b$ is non-zero. Hence,
we get a sublattice isomorphic to $\oplus_n k H$ for every integer $k \geq 1$.
\end{proof}

We are now ready to prove Theorems \ref{main/thm} and \ref{main-gen/thm}, which we state
again here below for the reader convenience.

\begin{statement}{Theorem 1.1}
Let $M$ be a closed connected oriented PL $4$-manifold. Then, there exists a branched
covering $p\: M \to N$ with:
\begin{itemize}
\item[\(a\)] $N = \CP^2$ $\Leftrightarrow$ $b_2^+(M) \geq 1;$
\item[\(b\)] $N = \CPbar^2$ $\Leftrightarrow$ $b_2^-(M) \geq 1;$
\item[\(c\)] $N = S^2 \simtimes S^2$ $\Leftrightarrow$ $b_2^+(M) \geq 1$ and $b_2^-(M) 
\geq 1;$
\item[\(d\)] $N = S^2 \times S^2$ $\Leftrightarrow$ $b_2^+(M) \geq 1$ and $b_2^-(M) \geq 
1;$
\item[\(e\)] $N = S^3 \times S^1$ $\Leftrightarrow$ $b_1(M) \geq 1.$
\end{itemize}
In all cases, we can assume that $p$ is a simple branched covering of degree $d\leq4$,
whose branch set $B_p$ is a closed locally flat PL surface self-transversally immersed in
$N$.
Moreover, $B_p$ can be desingularized to become embedded in $N$, with the following
estimates for the degree $d$: $d \leq 5$ in cases \(a\) and \(b\) for $b_2(M) \geq 2$ and
$\beta_M$ odd, case \(c\) for $\beta_M$ odd, case \(d\) for $\beta_M$ even, and case
\(e\); $d \leq 6$ in cases \(a\) and \(b\) for $b_2(M) \geq 2$ and $\beta_M$ even, case
\(c\) for $\beta_M$ even, and case \(d\) for $\beta_M$ odd; $d \leq 9$ in cases \(a\) and
\(b\) for $b_2(M) = 1$.
\end{statement}

\begin{proof}
First of all, we recall the well known fact that in a closed connected oriented PL
4-manifold $M$ any homology class $\alpha \in H_2(M)$ can be represented by a closed
oriented locally flat PL surface $F \subset M$ (see \cite{GS99} or \cite{Ki89}). Moreover,
$F$ can be easily made connected by embedded surgery. Similarly, any homology class
$\alpha \in H_3(M)$ can be represented by a closed oriented locally flat PL 3-manifold $N
\subset M$, but in this case $N$ can be made connected only if $\alpha$ is primitive (see
\cite{MP77}).

\(a\). Given any $d$-fold branched covering $p \: M \to \CP^2$, we can assume up to PL
isotopy that $B_p \subset \CP^2$ meets $\CP^1$ transversally. Then, $F = p^{-1}(\CP^1)
\subset M$ is a closed oriented locally flat PL surface, which represents a non-zero
element $\phi \in H_2(M) / \Tor H_2(M)$ such that $\beta_M(\phi,\phi) = d > 0$. Hence,
$b_2^+(M) \geq 1$.

\pagebreak

For the converse, assume that $b_2^+(M) \geq 1$. By the first part of Lemma
\ref{class1/thm}, there exists a class $\phi \in H_2(M) / \Tor H_2(M)$ such that
$\beta_M(\phi,\phi) = 4$. Then, the desired 4-fold branched covering $p \: M \to \CP^2$
can be obtained by applying Theorem \ref{bc2-cp2/thm} \(a\) in the case $d=4$ to any
closed connected oriented locally flat PL surface $F \subset M$ representing the homology
class $\phi$. In this way, the branch set $B_p$ turns out to be a closed locally flat PL
surface self-transversally immersed in $\CP^2$.

To obtain a non-singular branch surface according to the cases stated in the theorem, we
apply Theorem \ref{bc2-cp2/thm} \(a\) with $d\geq 5$ to any closed connected oriented
locally flat PL surface $F \subset M$ representing the homology class $\phi$ provided by
the corresponding cases of the first part of Lemma \ref{class1/thm} with $k=d$, taking
into account that $\beta_M$ is necessarily odd if $b_2(M) = 1$.

\(b\). This case immediately follows from case \(a\), by reversing the orientations.

\(c\) and \(d\). As in the proof of Theorem \ref{bc2-s2xs2/thm}, denote by $\xi \: E \to
S^2$ the bundle $S^2 \times S^2 \to S^2$ or $S^2 \simtimes S^2 \to S^2$, depending on the
case, and let $S^2_1\,,\,S^2_2 \subset E$ be any PL section and fiber of $\xi$,
respectively.

Given a branched $d$-fold covering $p \: M \to E$, we can assume up to PL isotopy that
$B_p \subset E$ meets both the surfaces $S^2_1$ and $S^2_2$ transversally. Then, $F_1 =
p^{-1}(S^2_1) \subset M$ and $F_2 = p^{-1}(S^2_2) \subset M$ are closed oriented locally
flat PL surfaces such that $F_1 \cdot F_2 = d > 0$ and $F_2 \cdot F_2 = 0$. It follows
that the homology class $\phi \in H_2(M)/ \Tor H_2(M)$ represented by $F_2$ is non-zero
and $\beta_M(\phi,\phi) = 0$. Therefore, $\beta_M$ is indefinite, hence $b_2^+(M) \geq 1$
and $b_2^-(M) \geq 1$.

Conversely, assuming $b_2^+(M) \geq 1$ and $b_2^-(M) \geq 1$, let $\phi_1,\phi_2 \in
H_2(M)/\Tor H_2(M)$ be the homology classes given by the second part of Lemma
\ref{class1/thm} with $n=0$ for $E = S^2 \times S^2$ or $n = 1$ for $E = S^2 \simtimes
S^2$, and $k = d$ depending on the case of the statement that we want to realize. We
represent $\phi_1$ and $\phi_2$ by closed connected oriented locally flat PL surfaces
$F_1, F_2 \subset M$, respectively, which can be assumed to be transversal to each other.
Then, we can perform an embedded surgery, without changing the homology classes of the
surfaces but increasing the genus of one of them, to eliminate each pair of opposite
intersection points (if any) between $F_1$ and $F_2$. This determines new surfaces $F_1$
and $F_2$ with $d$ transversal positive intersection points. At this point, the wanted
branched covering $p \: M \to E$ can be obtained by applying Theorem \ref{bc2-s2xs2/thm}
to $(M;F_1, F_2)$.

\(e\). Given any $d$-fold branched covering $p \: M \to S^3 \times S^1$, we can assume up
to PL isotopy that $B_p \subset S^3 \times S^1$ meets $S^3 \times \{\ast\}$ transversally
and is disjoint from $\{\ast\} \times S^1$. Then, $N = p^{-1}(S^3 \times \{\ast\}) \subset
M$ and $C = p^{-1}(\{\ast\} \times S^1) \subset M$ are closed oriented locally flat PL
submanifolds of dimensions $3$ and $1$, respectively, such that $N \cdot C = d > 0$. Then,
$C$ represents a non-trivial homology class in $H_1(M)/ \Tor H_1(M)$, and so $b_1(M) \geq
1$.

Conversely, if $b_1(M) \geq 1$, and hence $b_3(M) \geq 1$, let $N \subset M$ be a closed
connected oriented locally flat 3-manifold representing a primitive non-trivial element of
$H_3(M)$. Then $N$ does not disconnect $M$ and we can apply Theorem \ref{bc3/thm} \(b\) to
get the desired branched covering $p \: M \to S^3 \times S^1$.
\end{proof}

\begin{statement}{Theorem 1.3}
Let $M$ be a closed connected oriented PL 4-manifold and let $m$ and $n$ be non-negative
integers. Then, there exists a branched covering $p\: M \to N$ with:
\begin{itemize}
\item[\(a\)] $N = \cs_m\CP^2 \cs_n \CPbar^2$ $\Leftrightarrow$ $b_2^+(M) \geq m$ and 
$b_2^-(M) \geq n$;
\item[\(b\)] $N = \cs_n(S^2 \times S^2)$ $\Leftrightarrow$ $b_2^+(M)\geq n$ and $b_2^-(M) 
\geq n$;
\item[\(c\)] $N = \cs_n(S^3 \times S^1)$ $\Leftrightarrow$ $\pi_1(M)$ admits a free group 
of rank $n$ as a quotient.
\end{itemize}
In all cases, we can assume that $p$ is a simple branched covering of degree $d\leq4$,
whose branch set $B_p$ is a closed locally flat PL surface self-transversally immersed in
$N$. Moreover, $B_p$ can be desingularized to become embedded in $N$, with the following
estimates for the degree $d$: $d \leq 5$ in case \(a\) for $b_2(M) \geq 2(m+n)$ and
$\beta_M$ odd, case \(b\) for $\beta_M$ even, and case \(c\); $d \leq 6$ in case \(a\) for
$b_2(M) \geq 2(m+n)$ and $\beta_M$ even, and case \(b\) for $\beta_M$ odd; $d \leq 9$ in
case \(a\) for $b_2(M) < 2(m+n)$.
\end{statement}

\begin{proof}
We only sketch the proof, because it follows the same ideas of the proof of Theorem
\ref{main/thm}. For items \(a\) and \(b\) the implications to the right are
straightforward, so we only discuss the implications to the left.
	
\(a\). We consider the proper sublattice $\Lambda_{m,n}(d) \subset H_2(M)/ \Tor H_2(M)$
given by Lemma \ref{class2/thm}, according to the particular case of item \(a\) that we
want to prove, and represent the base of $\Lambda_{m,n}(d)$ by disjoint embedded oriented
connected locally flat PL surfaces $F_1, \dots, F_{m+n} \subset M$. We also consider
$\CP_1^1, \dots, \CP^1_{m+n} \subset N$, where $\CP^1_i$ is a projective line in the
$i$-th connected summand of $N = \cs_m\CP^2 \cs_n \CPbar^2$.

Next, we construct $d$-fold simple branched coverings $t_i \: (T_{F_i}; F_i) \to
(T_{\CP^1_i}; \CP^1_{i})$ between tubular neighborhoods, based on Proposition
\ref{bundle/thm} as in the proof of Theorem \ref{bc2-cp2/thm}, whose restrictions on the
boundary are ribbon fillable by Proposition \ref{bundleS/thm}. Now, we put
$$t = \cup_i t_i \: {\cup_i} (T_{F_i}; F_i) \to \cup_i (T_{\CP^1_i}; \CP^1_{i}),$$
$$W = \Cl(M - \cup_i T_{F_i}),$$
$$Y = \Cl(N - \cup_i T_{\CP^1_i}) \cong \cs_{m+n} B^4 \cong S^4 - \Int(B^4_1 \cup\dots 
\cup B^4_{m+n}).$$
\medskip
Then, we extend the ribbon fillable branched covering $t_{|\Bd} \: \Bd W \to \Bd Y \cong
\cup_{m+n} S^3$ to a simple branched covering $q \: W \to Y$ by means of Theorem
\ref{bc-ext/thm}, and finally we obtain the wanted branched covering by putting $p = q
\cup t \: M \to N$.

\(b\). By Lemma \ref{hyperb/thm}, we can find a sublattice of $(H_2(M) / \Tor H_2(M),
\beta_M)$ which is isomorphic to $\oplus_n d H$, where $d$ can be chosen according to the
specific case that we want to obtain. For each direct summand $d H$, we choose a basis
$\eta_i, \eta_i'$ of it such that $\beta_M(\eta_i, \eta_i) =\beta_M(\eta_i', \eta_i') = 0$
and $\beta_M(\eta_i, \eta_i') = d$, for $i = 1,\dots, n$.

Such homology classes can be represented by pairwise transversal closed connected oriented
locally flat PL surfaces $F_i, F_i'\subset M$ such that their geometric intersections
equal the algebraic ones, for $i = 1,\dots, n$.

Then, we can find a simple $d$-fold branched covering as desired by repeating the argument
used in case \(a\) (see also the proof of case \(d\) of Theorem \ref{main/thm}), with the
following setting
$$t = \cup_i t_i \: {\cup_i} (T_{F_i} \cup T_{F_i'}; F_i, F'_i) \to \cup_i 
\big(T_{S^2_{1i}} \cup T_{S^2_{2i}}; S^2_{1i}, S^2_{2i}\big),$$
$$W = \Cl\big(M - \cup_i (T_{F_i} \cup T_{F_i'})\big),$$
$$Y = \Cl\big(N - \cup_i (T_{S^2_{1i}} \cup T_{S^2_{2i}})\big) \cong \cs_{n} B^4 \cong S^4 
- \Int(B^4_1 \cup\dots \cup B^4_{n}).$$
\medskip
\(c\). Suppose that there is a $d$-fold branched covering $p \: M \to N=\cs_n(S^3 \times
S^1)$ for some $d \geq 1$. Let $\gamma_1, \dots, \gamma_n \in \pi_1(N) \cong \F_n$ be the
free generators, where $\F_n$ is the free group of rank $n$. By lifting loops, we can find
elements $\tilde\gamma_1,\dots, \tilde\gamma_n \in \pi_1(M)$ such that
$p_*(\tilde\gamma_i) = \gamma_i^{a_i}$ for certain $a_i \in\{1, \dots, d\}$ and for all $i
= 1,\dots, n$, where $p_*\: \pi_1(M) \to \pi_1(N)$ is the homomorphism induced by $p$.
Then, $p_*(\pi_1(M))$ contains the subgroup $\langle \gamma_1^{a_1}, \dots, \gamma_n^{a_n}
\rangle$ of $\F_n$. It follows that $p_*(\pi_1(M))$ is free of rank at least $n$, implying
that it admits $\F_n$ as a quotient.

For the converse, we observe that for every epimorphism $\phi \: \pi_1(M) \to \F_n$ there
exists a PL embedding $h \: {\vee_{\!n} S^1} \to M$ such that $h_*\: \pi_1({\vee_{\!n}
S^1}) \to \pi_1(M)$ is a right inverse of $\phi$. We want to define a PL map $g \: M \to
{\vee_{\!n} S^1}$, which is a left inverse of $h$, and such that $g_* = \phi \: \pi_1(M)
\to \pi_1({\vee_{\!n} S^1}) \cong \F_n$. To define $g$, we consider a handlebody
decomposition of $M$ with only one 0-handle $H^0$ centered at $h(*)$, where $*$ is the
join point of ${\vee_{\!n} S^1}$, and such that $H^0 \cup H^1_1 \cup \dots \cup H^1_n$ is
a regular neighborhood of $h({\vee_{\!n} S^1})$, for some 1-handles $H^1_1, \dots,
H^1_n$.\break At this point, the construction of $g$ is as follows: over $H^0 \cup H^1_1
\cup \dots \cup H^1_n$, the map $g$ is a PL collapsing retraction over $h({\vee_{\!n}
S^1})$ composed with $h^{-1} \: h({\vee_{\!n} S^1}) \to {\vee_{\!n} S^1}$; over the
remaining 1-handles it is defined according to $\phi$; then $g$ can be extended over the
2-handles, thanks to the compatibility with $\phi$ over the generators of $\pi_1(M)$;
finally, there is no obstruction to further extend $g$ over the higher index handles.

For every $i = 1\dots,n$, let $y_i$ be a point in the $i$-th component of $\vee_{\!n} S^1
- \{*\}$, over which $g$ is transversal ($y_i$ is a regular value), and let $Y_i$ be the
connected component of $g^{-1}(y_i)$ that contains $h(y_i)$. Then, $Y_i$ is a connected
orientable locally flat PL 3-manifold in $M$.

Let $M'$ be $M$ cut open along $Y_1, \dots, Y_n$. By construction, $M'$ is a connected
4-manifold with $2n$ boundary components $Y_1, \overline Y_1, \dots, Y_n, \overline Y_n$
and there are identifications $Y_i \cong \overline Y_i$ coming from the cuts. By Theorem
\ref{bc-ext/thm} there exists a simple $d$-fold branched covering $q \: M' \to S^4 -
\cup_{i=1}^n \Int(B^4_i \cup \Bbar^4_i)$ such that the coverings $q_{|Y_i} \: Y_i \to \Bd
B_i$ and $q_{|\overline Y_i} \: \overline Y_i \to \Bd \Bbar_i^4$ match with respect to the
above identifications, where $B^4_i$ and $\Bbar^4_i$ are disjoint 4-balls in $S^4$, for
$i=1,\dots, n$. We can assume that $B_q$ is a locally flat self-transversally immersed
compact PL surface if $d \geq 4$, and that it is embedded if $d \geq 5$. Then, we can glue
back $Y_i$ with $\overline Y_i$, as well as $\Bd B^4_i$ with $\Bd \Bbar^4_i$, by means of
the identifications needed to reconstruct $M$ and $\cs_n(S^3 \times S^1)$ respectively.
Then we get a simple branched covering $p \: M \to \cs_n (S^3 \times S^1)$ as desired.
\end{proof}

\section{Final remarks}

In Theorem \ref{main/thm} \(a\), the simple branched covering $p \: M \to \CP^2$ can be 
constructed such that $p^*(w_2(\CP^2)) = w_2(M)$ if $w_2(M)^2 \neq 0$ in $H^4(M;\Z_2)$. 
Indeed, in the proof it is enough to take as $\phi \in H_2(M)/ \Tor H_2(M)$ the Poincar\'e 
dual (modulo $\Tor H_2(M)$) of any integral lift of $w_2(M)$ with positive (odd) square. 
An analogous fact holds for Theorem \ref{main/thm} \(b\).

In Theorem \ref{main-gen/thm} \(b\), for $\beta_M$ odd and $n < \max(b_2^+(M), b_2^-(M))$,
we can also obtain a 5-fold simple covering $p \: M \to \cs_n(S^2 \times S^2)$ branched
over a non-singular PL surface, by using the last part of Lemma \ref{hyperb/thm} and
taking $d = 5$ in the proof.

In Theorem \ref{bc3/thm} we can take $p$ such that its restriction $p_{|N}$ coincides with
any given ribbon fillable $d$-fold branched covering $c \: N \to S^3$. Indeed, in the
proof the choice of $c$ as such a covering is arbitrary.

The following Corollary to Theorem \ref{bc3/thm} is immediate but possibly interesting for
the PL or smooth Schoenflies Conjecture in $S^4$.

\begin{corollary}\label{bc3s/thm} 
Let $\Sigma^3 \subset S^4$ be a PL embedded $3$-sphere and let $d \geq 4$. Then, there
exists a $d$-fold simple covering $p \: (S^4; \Sigma^3) \to (S^4; S^3)$ branched over a
locally flat PL self-transversally immersed surface, which can be taken embedded for $d
\geq 5$. Moreover, the restriction $p_{|\Sigma^3} \: \Sigma^3 \to S^3$ can be arbitrarily
chosen among $d$-fold ribbon fillable branched coverings.
\end{corollary}

Moreover, for a PL 3-manifold $N \subset M$, one can prove that there is a simple branched
covering $p \: (M; N) \to (S^4; S^3)$ even though $N$ does not disconnect $M$. In this
case, we obtain an arbitrary degree $d \geq 6$ and a locally flat PL embedded branch
surface. The proof goes as follows: following the proof of Theorem \ref{bc3/thm} \(b\), we
begin with a ribbon fillable branched covering $c \: \Bd M_1 \to S^3$ of degree $d \geq
6$, with $M_1$ a collar of $N$ in $M$. This is possible because $\Bd M_1$ has two
connected components homeomorphic to $N$. Then, by Theorem \ref{bc-ext/thm}, there are two
extensions of $c$ as simple $d$-fold branched coverings $p_1 \: M_1 \to S^4_-$ and $p_2 \:
M_2 \to S^4_+$, both branched over a locally flat properly embedded PL surface. Their
union provides the desired branched covering $p \: (M; N) \to (S^4; S^3)$.

In Theorem \ref{bc2/thm}, for $k \geq 2$, we can take $S^2 \subset S^4$ instead of $T_k$,
with $d = 4k$. Thus, there exists a simple branched covering $p \: (M; F) \to (S^4; S^2)$
even though $F$ is not connected. The proof is essentially the same, the only difference
consisting in the identification of the base of $c_i \: N_i \to B^3$ with a single copy of
$B^3 \subset S^4$ instead of $k$ copies of it.

The singularities of the branch surfaces of all the 4-dimensional simple branched
coverings we have constructed, namely the transversal self-intersections, originate from
the application of Theorem \ref{bc-ext/thm}, which was proved in \cite{PZ16}. In the
construction therein, such singularities appear in pairs, so one can investigate to what
extent they can be eliminated, without increasing the covering degree. Then, we conclude
by asking the following question (cf. Problem 4.113 (A) in Kirby's list \cite{Ki95}).

\begin{question}
Can the simple branched covering $p \: M \to N$ in Theorem \ref{main/thm} be always chosen 
with a locally flat PL embedded branch surface even for $d = 4$?
\end{question}


\end{document}